\documentclass[onefignum,onetabnum]{siamart171218}

\usepackage{xcolor}
\usepackage{amsfonts}
\usepackage{graphicx}
\usepackage{epstopdf}
\usepackage{overpic}
\usepackage{amsmath,esint}
\usepackage{epsfig}
\usepackage{todonotes}
\usepackage{tikz}
\usetikzlibrary{arrows}
\usetikzlibrary{fadings}
\usepackage{pgfplots}
\pgfplotsset{compat=1.10}
\usepgfplotslibrary{fillbetween}
\usetikzlibrary{patterns}
\usepackage{enumerate}
\usepackage{enumitem}
\newsiamremark{remark}{Remark}

\title{Bounding Zolotarev numbers using Faber rational functions\thanks{Funding: This work is partly supported by National Science Foundation grant no.~1818757 and no.~DGE-1650441.}}

\author{Daniel Rubin\thanks{(\email{danielrubin64@gmail.com}).}
\and Alex Townsend\thanks{Mathematics Department, Cornell University, Ithaca, NY 14853-4201, United States (\email{townsend@cornell.edu}).} \and Heather Wilber\thanks{Center for Applied Mathematics, Cornell University, Ithaca, NY 14853-4201, United States (\email{hdw27@cornell.edu}).}}

\begin{document}
\maketitle
\begin{abstract}
By closely following a construction by Ganelius, we use Faber rational functions to derive tight explicit bounds on Zolotarev numbers.  We use our results to bound the singular values of matrices, including complex-valued Cauchy matrices and Vandermonde matrices with nodes inside the unit disk. We construct Faber rational functions using doubly-connected conformal maps and use their zeros and poles to supply shift parameters in the alternating direction implicit method.  
\end{abstract}



\section{Introduction}\label{sec:Introduction} 
The Zolotarev number from rational approximation theory is given by~\cite{Zolotarev_1877_01}
\begin{equation} 
Z_n(E,F) = \inf_{s_n\in\mathcal{R}_{n,n}} \frac{\sup_{z\in E} |s_n(z)|}{\inf_{z\in F} |s_n(z)|}, \qquad n\geq 0,
\label{eq:Zolotarev} 
\end{equation} 
where $\mathcal{R}_{n,n}$ denotes the set of rational functions of type $(n,n)$ and $E, F\subset \mathbb{C}$ are disjoint sets in the complex plane. Due to the infimum over $\mathcal{R}_{n,n}$ in~\cref{eq:Zolotarev}, we know that $Z_n(E,F)\leq \sup_{z\in E} |s_n(z)| / \inf_{z\in F} |s_n(z)|$ for any $s_n\in\mathcal{R}_{n,n}$. In this paper, we closely follow a construction by Ganelius~\cite{Ga} to derive Faber rational functions and use them to derive explicit upper bounds on $Z_n(E,F)$ when $E$ and $F$ are such that $\mathbb{C}\setminus F$ is open and simply connected and $E$ is a compact, simply-connected subset of $\mathbb{C}\setminus F$.  Throughout this paper, we assume that the boundaries of $E$ and $F$ are rectifiable Jordan curves. To be concrete, the main situation we focus on is when 
\begin{itemize} 
\item[(A1)] $E$ and $F$ are disjoint, simply-connected, compact sets (see~\cref{fig:SimpleConstruction1}).
\end{itemize} 
 In~\cref{sec:Extensions}, we discuss two other types of sets $E$ and $F$: (A2) $\mathbb{C}\setminus F$ is a bounded domain containing $E$ (see~\cref{fig:SimpleConstruction}) and (A3) $F$ is an unbounded domain and $E$ is a compact domain contained in $\mathbb{C}\setminus F$ (see~\cref{fig:SimpleConstruction3}). 

\begin{figure} 
\centering
\tikzfading[name=fade inside, inner color=transparent!0,outer color=transparent!50]
\begin{tikzpicture}
\draw [thick, black, fill=teal] plot [smooth cycle] coordinates {(0,0) (1,1) (2,1) (1,0) (1.5,-1)};
\draw [thick, black, xshift=2cm,fill=cyan] plot [smooth cycle] coordinates {(0,0) (1,1) (2,2) (3,0)};
\draw [thick, black,fill=black] (9,1.6) circle [radius=.05];
\node at (.9,.5) {$E$};
\node at (3.7,.5) {$F$};
\node at (1,2) {$\Omega = \mathbb{C} \setminus (E\cup F)$};
\node at (9.3,1.6) {$\omega_0$};
\draw [->,black,thick] (5,2) to [out=30,in=150] (7,2);
\draw [<-,black,thick] (5,-1.5) to [out=-30,in=-150] (7,-1.5);
\node at (6,-1.5) {$\Psi$};
\node at (6,2.5) {$\Phi$};
\node at (9.5,-1) {$A$};
\fill [cyan,even odd rule,xshift=9.5cm,fill opacity=1,yshift=.3cm] (0,0) circle[radius=1.7cm] circle[radius=1.75 cm];
\fill [cyan,even odd rule,xshift=9.5cm,fill opacity=.9373,yshift=.3cm] (0,0) circle[radius=1.75cm] circle[radius=1.8 cm];
\fill [cyan,even odd rule,xshift=9.5cm,fill opacity=.875,yshift=.3cm] (0,0) circle[radius=1.8cm] circle[radius=1.85 cm];
\fill [cyan,even odd rule,xshift=9.5cm,fill opacity=.8125,yshift=.3cm] (0,0) circle[radius=1.85cm] circle[radius=1.9 cm];
\fill [cyan,even odd rule,xshift=9.5cm,fill opacity=.75,yshift=.3cm] (0,0) circle[radius=1.9cm] circle[radius=1.95 cm];
\fill [cyan,even odd rule,xshift=9.5cm,fill opacity=.6875,yshift=.3cm] (0,0) circle[radius=1.95cm] circle[radius=2 cm];
\fill [cyan,even odd rule,xshift=9.5cm,fill opacity=.625,yshift=.3cm] (0,0) circle[radius=2cm] circle[radius=2.05 cm];
\fill [cyan,even odd rule,xshift=9.5cm,fill opacity=.5625,yshift=.3cm] (0,0) circle[radius=2.05cm] circle[radius=2.1 cm];
\fill [cyan,even odd rule,xshift=9.5cm,fill opacity=.5,yshift=.3cm] (0,0) circle[radius=2.1cm] circle[radius=2.15 cm];
\fill [cyan,even odd rule,xshift=9.5cm,fill opacity=.4375,yshift=.3cm] (0,0) circle[radius=2.15cm] circle[radius=2.2 cm];
\fill [cyan,even odd rule,xshift=9.5cm,fill opacity=.375,yshift=.3cm] (0,0) circle[radius=2.2cm] circle[radius=2.25 cm];
\fill [cyan,even odd rule,xshift=9.5cm,fill opacity=.3125,yshift=.3cm] (0,0) circle[radius=2.25cm] circle[radius=2.3 cm];
\fill [cyan,even odd rule,xshift=9.5cm,fill opacity=.25,yshift=.3cm] (0,0) circle[radius=2.3cm] circle[radius=2.35 cm];
\fill [cyan,even odd rule,xshift=9.5cm,fill opacity=.1875,yshift=.3cm] (0,0) circle[radius=2.35cm] circle[radius=2.4 cm];
\fill [cyan,even odd rule,xshift=9.5cm,fill opacity=.125,yshift=.3cm] (0,0) circle[radius=2.4cm] circle[radius=2.45 cm];
\fill [cyan,even odd rule,xshift=9.5cm,fill opacity=0.0625,yshift=.3cm] (0,0) circle[radius=2.45cm] circle[radius=2.5 cm];
\draw [thick, black, xshift=9.5cm,yshift=.3cm] (0,0) circle [radius=1.7];
\draw [thick, black, xshift=9.5cm,fill=teal,yshift=.3cm] (0,0) circle [radius=1];
\end{tikzpicture}
\caption{We mainly focus on the situation when $E$ and $F$ are disjoint and compact sets in the complex plane. Here, $\Phi: \Omega \rightarrow A$ is the conformal map that transplants $\Omega$ onto an annulus $A = \{z\in\mathbb{C} : 1< |z|<h\}$ with $h = \exp(1/{\rm cap}(E,F))$.   The location $\omega_0\in\mathbb{C}$ is the pole of the inverse map $\Psi = \Phi^{-1}$. }
\label{fig:SimpleConstruction1} 
\end{figure}
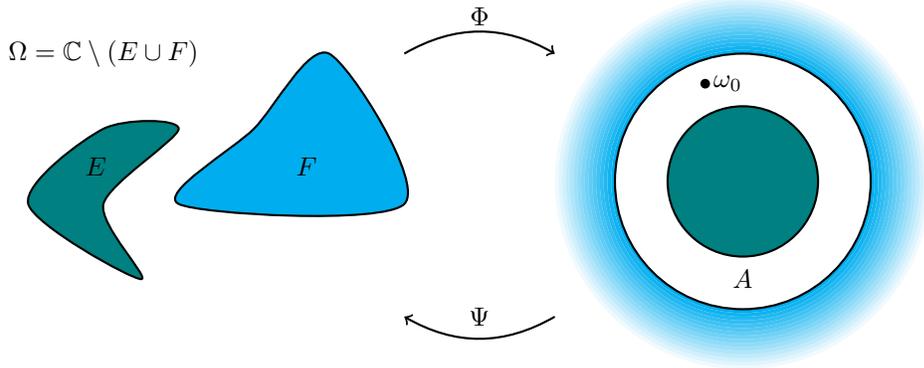 

The Zolotarev number, $Z_n(E,F)$, has applications for explicitly bounding the singular values of matrices~\cite{Beckermann_19_01}, solving Sylvester matrix equations~\cite{Lebedev_77_01}, the computation of the singular value decomposition of a matrix~\cite{Nakatsukasa_16_01}, and the solution of generalized eigenproblems~\cite{Guttel_14_01}. It is often important to have a tight explicit bound as well as the zeros and poles of a rational function that attains the bound. Explicit and tight bounds on $Z_n(E,F)$ are available in the literature when: (i) $E$ and $F$ are disjoint intervals~\cite[Sec.~3.2]{Beckermann_19_01} and (ii) $E$ and $F$ are disjoint disks~\cite{Starke_92_01}.
Faber rationals offer a general approach for obtaining explicit bounds.

It is immediate that $Z_0(E,F) =1$ and $Z_{n+1}(E,F)\leq Z_n(E,F)$ for $n\geq 0$. As a general rule, the number $Z_n(E,F)\rightarrow 0$ rapidly as $n\rightarrow \infty$ if $E$ and $F$ are disjoint, compact, and well-separated.  More precisely, for disjoint sets $E$ and $F$, a lower bound on $Z_n(E,F)$ as well as its asymptotic behavior is known~\cite{G}:
\begin{equation} 
Z_n(E,F) \geq h^{-n}, \qquad \lim_{n\rightarrow \infty} \left(Z_n(E,F)\right)^{1/n} = h^{-1},\qquad h=\exp\left(\frac{1}{{\rm cap}(E,F)}\right),
\label{eq:lower} 
\end{equation} 
where ${\rm cap}(E,F)$ is the capacity of a condenser with plates $E$ and $F$~\cite[Thm.~VIII.~3.5]{Saff_13_01}. Our goal is to derive explicit upper bounds on Zolotarev numbers of the form: 
\[
Z_n(E,F) \leq K_{E,F} h^{-n}, \qquad n\geq 0,
\]
where $K_{E,F}$ is a constant that depends on the geometry of $E$ and $F$. When $E$ and $F$ are disjoint disks, it is known that $K_{E,F}$ can be taken to be $1$~\cite{Starke_92_01} (see~\cref{sec:Mobius}) and when $E$ and $F$ are disjoint real intervals, $K_{E,F}$ can be taken to be $4$~\cite{Beckermann_19_01}. To the authors' knowledge, the best previous explicit upper bound for sets satisfying (A1)-(A3)  is $Z_n(E,F)\leq 4000n^2 h^{-n}$~\cite{Ga}. 

\subsection{The total rotation of a domain}
Our upper bound on $Z_n(E,F)$ involves the so-called total rotation of the domains $E$ and $F$~\cite{Gaier_87_01, Radon_1919}. 
\begin{definition} 
Let $E\subset \mathbb{C}$ be a simply-connected domain with a rectifiable Jordan curve boundary of length $1$. The total rotation of $E$ is defined as 
\begin{equation}
{\rm Rot}(E) = \frac{1}{2\pi}\int_{\partial E} \left|d\theta(s)\right|,
\label{eq:TotalRotation}
\end{equation} 
where $\theta(s)$ for $s\in(0,1)$ is the angle of the boundary tangent of $E$ (which exists for almost every $s\in(0,1)$). 
\label{def:TotalRotation} 
\end{definition} 
If $\theta(s)$ can be extended to a function of bounded variation to $s\in[0,1]$, then ${\rm Rot}(E)<\infty$. 
For any simply-connected domain, we note that ${\rm Rot}(E) \geq 1$. When $E$ is a polygon, $2 \pi {\rm Rot}(E)$ equals the sum of the absolute values of $E$'s exterior angles.  Moreover, when $E$ is a convex domain, ${\rm Rot}(E) = 1$~\cite[p.~6]{Anderson_1984}. 
\subsection{Main Theorem} 
We are now ready to state our main theorem. 
\begin{theorem}[Main Theorem]
Let $E,F\subset \mathbb{C}$ be disjoint, simply-connected, compact sets with rectifiable Jordan boundaries. Then, for $h=\exp(1/{\rm cap}(E,F))$, we have
\[
Z_n(E,F) \leq (2{\rm Rot}(E)+1)(2{\rm Rot}(F)+1)h^{-n} + \mathcal{O}(h^{-2n}), \qquad \text{as } n\rightarrow \infty,
\] 
where ${\rm Rot}(E)$ and ${\rm Rot}(F)$ are the total rotation of the boundaries of the domains $E$ and $F$, respectively.
If, in addition, $E$ and $F$ are convex sets, then we simply have $Z_n(E,F) \leq 9h^{-n} + \mathcal{O}(h^{-2n})$ as $n\rightarrow \infty$. 
\label{thm:MainTheorem} 
\end{theorem} 

The rational function that we use to derive our upper bound in~\cref{eq:FinalBound} below is the so-called Faber rational function associated with the sets $E$ and $F$ (see~\cref{eq:FaberRationalFunction}). 
 \Cref{thm:MainTheorem} shows that the lower bound on $Z_n(E,F)$ in~\cref{eq:lower} is sharp up to a constant. In particular, for disjoint, simply-connected, compact sets $E,F\subset \mathbb{C}$ with rectifiable Jordan boundaries, we have
\[
1 \leq \lim_{n\rightarrow \infty} \frac{Z_n(E,F)}{h^{-n}} \leq (2{\rm Rot}(E)+1)(2{\rm Rot}(F)+1).
\]
We do not believe that the upper bound of $(2{\rm Rot}(E)+1)(2{\rm Rot}(F)+1)$ is sharp and highlight in our derivation where there might be possible improvements.  

The actual upper bound that we derive is an inelegant expression that is nevertheless explicit and computable. With the same assumptions as in~\cref{thm:MainTheorem}, we have
\begin{equation} 
Z_n(E,F) \leq \!\!\left(\! \frac{ \frac{M_n(E,F)M_n(F,E)}{1-h^{-2n}} + \frac{32n M_n(E,F)h^{-n}}{(1 - (1+M_n(E,F))h^{-n})^2}}{\max\!\left\{0,1 - \frac{M_n(E,F)M_n(F,E)}{1-h^{-2n}}h^{-n} - \frac{M_n(E,F)}{1 - (1+M_n(E,F))h^{-n}}h^{-n} - h^{-2n}\right\}}\!\right)\!h^{-n}
\label{eq:FinalBound} 
\end{equation} 
for any $n> N_0$. Here, $N_0 = \max \{ 1 + 1/(h-1), \log(x_0)/\log(h)\}$, $M_n(E,F) = 2{\rm Rot}(E) + 2h^{-n}{\rm Rot}(F) + h^{-n} + 1$, and 
\[
x_0 = {\rm Rot}(E) + 1 + \sqrt{ ({\rm Rot}(E) + 1)^2 + 2{\rm Rot}(F)+1}.
\]
Note that for sufficiently large $n$, the denominator in~\cref{eq:FinalBound} becomes strictly positive and takes the limiting value of $1$ as $n\rightarrow \infty$. In particular, the bound in~\cref{eq:FinalBound} has the correct geometric decay to zero as $n\rightarrow\infty$. If the bound in~\cref{eq:FinalBound} turns out to $>1$ (which it may for small $n$), then one is welcome to take $Z_n(E,F) \leq 1$ instead. If, in addition, $E$ and $F$ are convex sets, then~\cref{eq:FinalBound} can be simplified as $x_0 = 2 + \sqrt{7}$ and $M_n(E,F) = M_n(F,E) = 3(1+h^{-n})$. \Cref{fig:bounds} illustrates how the bounds behave for convex sets with varying values of $h$. 

\begin{figure} 
\label{fig:bounds}
\centering 
\begin{minipage}{.65\textwidth}
\begin{overpic}[width=\textwidth,trim= 1cm 1cm  1cm 1cm,clip=true]{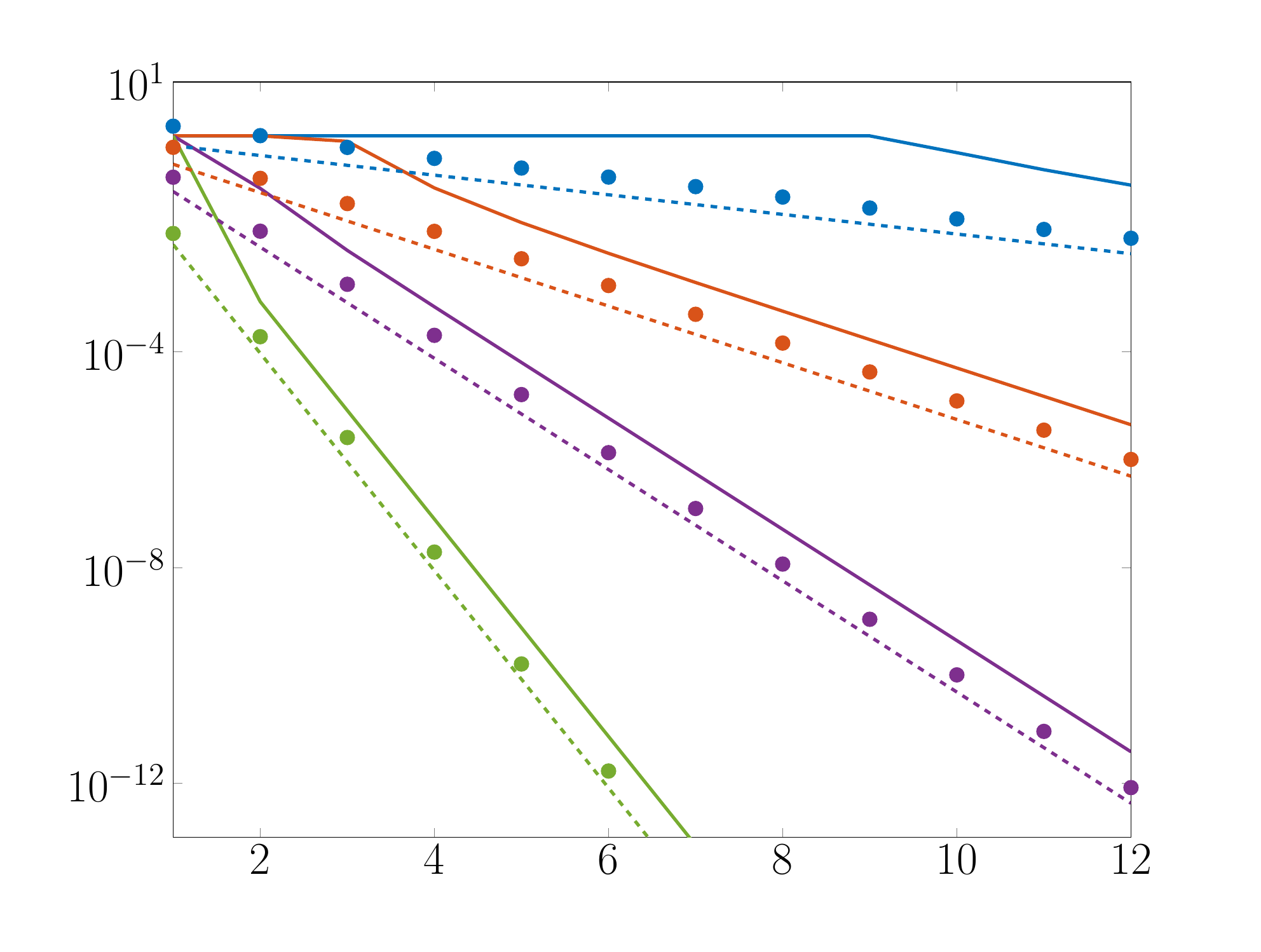}
\put(55, -2){$n$}
\put(-5, 31){\rotatebox{90}{magnitude}}
\put(32, 20){\rotatebox{-50}{{\small{$h \approx 103.3$}}}}
\put(60, 25){\rotatebox{-38}{{\small{$h \approx 10.7$}}}}
\put(70, 40){\rotatebox{-19}{{\small{$h \approx 3.4$}}}}
\put(70, 55){\rotatebox{-7}{{\small{$h \approx 1.5$}}}}
\end{overpic} 
\end{minipage} 
\caption{Bounds on $Z_{n}(E_{\alpha}, -E_{\alpha})$ with $E_{\alpha} = \{ z  \in \mathbb{C} : \, Re(z) \in [-.4-\alpha,  .4 -\alpha], \:  Im(z) \in [-.6, .6] \}$, where $\alpha= .45$ (blue), $.6$ (orange), $1$ (purple), $3$ (green). As $\alpha$ grows,  $h = \exp(-1/{\rm cap}(E_{\alpha}, -E_{\alpha}))$ grows, and $Z_n(E_{\alpha}, -E_{\alpha})$ decays more rapidly.  The solid lines are the bounds from~\cref{thm:MainTheorem}, combined with the trivial bound  $Z_n(E, F)\leq 1$. The dotted lines are the lower bounds of $Z_n(E_\alpha,-E_\alpha) \geq h^{-n}$ in~\cref{eq:lower}. The dots are computed by first constructing the Faber rational function $r_n(z)$ associated with $(E_{\alpha}, -E_{\alpha})$,  and then computing the $\max_{z \in E_{\alpha}}|r_n(z)| / \min_{z \in -E_{\alpha}}|r_n(z)|$. }
\end{figure}
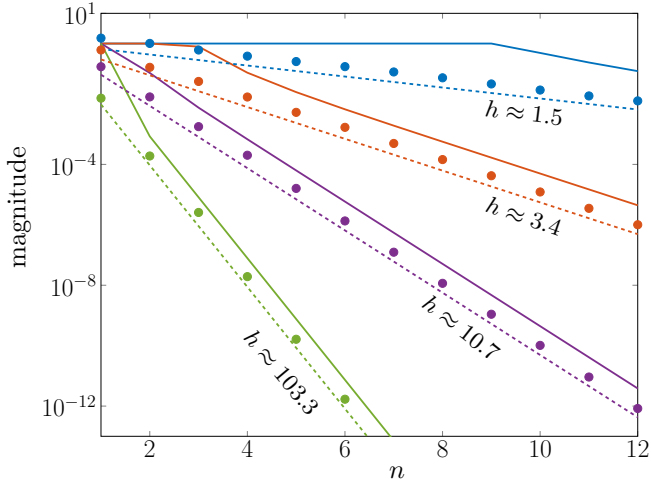

\subsection{Conformal mapping of doubly connected sets} 
\label{sec:ConformalMap}
The construction of Faber rationals requires conformal maps of doubly connected sets. 
A domain $\Omega\subset \mathbb{C}$ is said to be doubly connected if between any two points in $\Omega$ there are two different paths, i.e., two paths that cannot be smoothly deformed into each other. Any doubly connected domain, except for a punctured disk and a punctured plane, is conformally equivalent to  $A = \{ z\in\mathbb{C} : 1<|z|< h \}$ for some $h>1$~\cite[Ch.1,~sec.7]{courant1977}.  When $E,F\subset \mathbb{C}$ are as in~\cref{thm:MainTheorem}, $\Omega = \mathbb{C}\setminus (E\cup F)$ is doubly connected and can be conformally mapped to an annulus, i.e., 
\begin{equation}
\Phi : \Omega \rightarrow A,\qquad A = \{ z\in\mathbb{C} : 1<|z|< h \}.
\label{eq:map} 
\end{equation}
Since conformal maps preserve the logarithmic capacity of two plate condensers and the capacity of $A$ is $1/\log(h)$~\cite{Henrici}, the outer radius in~\cref{eq:map} is $h = \exp(1/{\rm cap}(E,F))$ (see~\cref{eq:lower}).  If $E$ and $F$ are disjoint polygons, then the conformal map, $\Phi$, can be constructed as a doubly-connected Schwarz--Christoffel mapping~\cite{Hu_98_01}. The inverse conformal map is denoted by $\Psi = \Phi^{-1}:A\rightarrow\Omega$.  

\subsection{Paper summary} 
This paper is structured as follows. In~\cref{sec:Mobius}, we briefly describe the simplest case when the conformal map $\Phi$ is a M\"{o}bius transform. In~\cref{sec:FaberRationals}, we describe the general construction of a Faber rational function associated to sets $E$ and $F$ satisfying the assumptions in~\cref{thm:MainTheorem}. In~\cref{sec:TheBound}, we bound Faber rational functions to obtain the explicit upper bound on $Z_n(E,F)$ given in~\cref{eq:FinalBound}. We extend our results to cases (A2) and (A3) in~\cref{sec:Extensions}. In~\cref{sec:Numeric} we provide numerical details, and we provide some examples of applications in~\cref{sec:Applications}. 

\section{When $\mathbf{\Phi}$ is a M\"{o}bius transformation}\label{sec:Mobius} 
Suppose that $E,F\subset\mathbb{C}$ are such that there exists a M\"{o}bius transform $\mathbf{\Phi}:\Omega \rightarrow A$ in~\cref{eq:map}. Since $\mathbf{\Phi}(z) = (az+b)/(cz+d)$ with $ad-bc\neq0$, we find that $\Phi$ is a type $(1,1)$ rational function. One can immediately verify that $\Phi^n\in \mathcal{R}_{n,n}$, $|\Phi(z)| \leq 1$ for $z\in E$, and $|\Phi(z)| \geq h$ for $z\in F$. This means that $Z_n(E,F) = h^{-n}$ as 
\[
h^{-n} \leq Z_n(E,F) \leq \frac{\sup_{z\in E} \left|\Phi^n(z)\right| }{\inf_{z\in F} \left|\Phi^n(z)\right|} \leq h^{-n},
\]  
where the lower bound is from~\cref{eq:lower}. Moreover, the rational function that attains the value of $Z_n(E,F)$ is known because it is simply given by $\Phi^n$. (An alternative proof of the optimality of $\Phi^n$ for $Z_n(E,F)$ is the near-circularity criterion~\cite{Starke_92_01}.)

When $E$ and $F$ are disjoint disks, there is a M\"{o}bius transform that maps $\Omega$ to an annlulus~\cite{Starke_92_01}. For example, suppose that $E = \{z\in \mathbb{C} : |z-z_0|\leq \eta_0\}$ and $F = -E$ with $0<\eta_0 < z_0$ and $z_0,\eta_0\in\mathbb{R}$. Then, the M\"{o}bius transform 
\[
\Phi(z) = \frac{z_0+\eta_0+c}{z_0+\eta_0-c} \frac{z-c}{z+c}, \qquad c = \sqrt{z_0^2-\eta_0^2}
\]
maps $\Omega = \mathbb{C}\setminus (E\cup F)$ onto the annulus $A = \{ z\in\mathbb{C} : 1<|z|< h\}$ with $h = (z_0+c)/(z_0-c)$. Therefore, we know that $Z_n(E,F)=(z_0-c)^n/(z_0+c)^{n}$ and this value is attained by the rational function $r_n(z) = \Phi^n(z)$ (see~\cref{fig:FaberDisks}). 

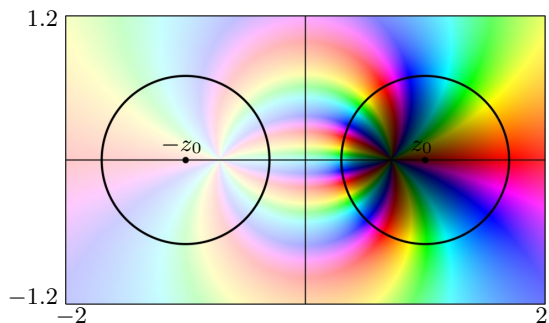
\begin{figure} 
\centering 
\begin{minipage}{.49\textwidth}
\begin{overpic}[width=\textwidth]{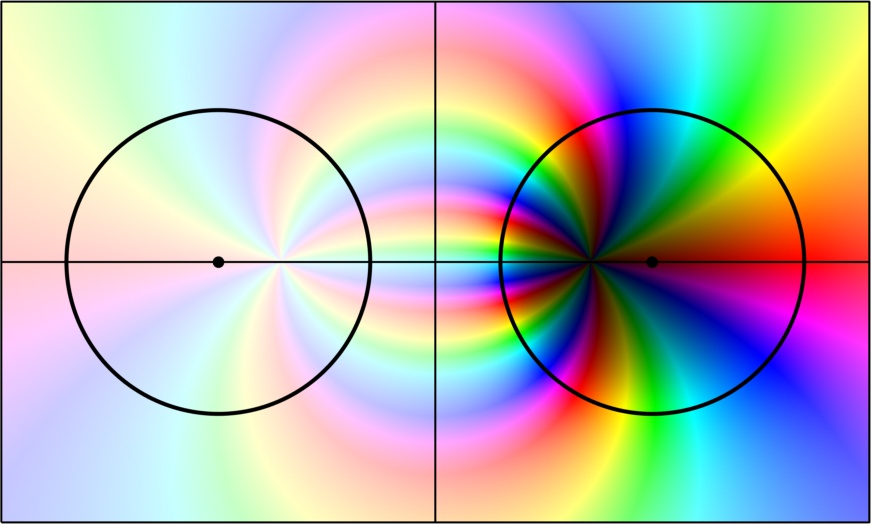}
\put(72, 32){\small{$z_0$}}
\put(20, 32){\small{$-z_0$}}
\put(-2,-4) {\small{$-2$}}
\put(98,-4) {\small{$2$}}
\put(-12,0) {\small{$-1.2$}}
\put(-8,58) {\small{$1.2$}}
\end{overpic} 
\end{minipage} 
\caption{A phase plot of the type $(5,5)$ Faber rational function on two disjoint disks, \smash{$E = \{z\in \mathbb{C} : |z-1|\leq .7\}$} and $F = -E$. Since there is a M\"{o}bius transformation from $\mathbb{C}\setminus (E \cup F)$ to an annulus, \smash{$Z_n(E,F)$} is known explicitly.}
\label{fig:FaberDisks}
\end{figure} 

\section{Constructing Faber rational functions}\label{sec:FaberRationals}

When $\Phi$ in~\cref{eq:map} is not a M\"{o}bius transform, we find that $\Phi^n\not \in \mathcal{R}_{n,n}$.  Therefore, $\Phi^n$ is not immediately useful for bounding $Z_n(E,F)$. However, we still expect $\Phi^n$ to be $\mathcal{O}(h^{n})$ near $F$ and $\mathcal{O}(1)$ near $E$. Thus, the idea is to construct a rational function from $\Phi^n$ by ``filtering" $\Phi^n$ using the Faber operator associated with $\Psi = \Phi^{-1}$~\cite{Anderson_1984}.  The rational function obtained from $\Phi^n$ after the ``filtering" process is called the {\em Faber rational associated with $E$ and $F$}. 

We now describe how one constructs a Faber rational, which closely follows the procedure in~\cite{Ga}. There are two main steps: (1) Constructing a function, $R_n(z)$, defined on $\mathbb{C}\setminus F$ with precisely $n$ zeros and (2) Constructing a rational function, $r_n(z)$, of type $(n,n)$. Both steps are accomplished by taking Cauchy integrals along the boundaries of $E$ and $F$. 

\subsection{Step 1: Constructing a function $\mathbf{R_n(z)}$ with $\mathbf{n}$ zeros near $\mathbf{E}$}\label{eq:ConstructRn} 
Let $\gamma :[0,1]\rightarrow \Omega$ be a positively oriented parameterization of the boundary $E$.  We can define the following ``filtered" function inside $E$: 
\begin{equation}
R_n(z) := \frac{1}{2\pi i} \int_{\gamma} \frac{ \Phi^n(\zeta) d\zeta}{\zeta - z}, \qquad z\in E. 
\label{eq:mainDef} 
\end{equation}
Intuitively, one would expect that $|R_n(z)|\approx |\Phi^n(z)|$ for $z\in E$. In particular, since $|\Phi^n(z)|\leq 1$ for $z\in E$ the function $|R_n(z)|$ should be relatively small on $E$. We use an idea that is similar to that found in~\cite{A}. 

\begin{lemma} 
Let $E,F\in\mathbb{C}$ be sets satisfying the assumptions in~\cref{thm:MainTheorem}. For $n\geq 1$, the function in~\cref{eq:mainDef} satisfies
\[
\sup_{z\in E} \left|R_n(z)\right| \leq \frac{M_n(E,F)}{1-h^{-2n}}, \qquad M_n(E,F) = 2{\rm Rot}(E) + 2h^{-n} {\rm Rot}(F) + 1 + h^{-n},
\]
where ${\rm Rot}(E)$ and  ${\rm Rot}(F)$ are defined in~\cref{eq:TotalRotation} and $h$ is defined in~\cref{eq:lower}. 
\label{lem:boundRn2}
\end{lemma} 
\begin{proof} 
Let $\Psi : A\rightarrow \Omega$ be the inverse conformal map to $\Phi$, which is meromorphic inside $A$ with a simple pole at $\omega_0$ for some $1<|\omega_0|<h$ (see~\cref{fig:SimpleConstruction1}). For any $z \in E$, we can use a change of variables to write
\[
R_n(z) = \frac{1}{2\pi i} \int_\gamma \frac{ \Phi^n(\zeta) d\zeta}{\zeta - z} = \frac{1}{2\pi i} \int_{|\omega| = \rho} \frac{\omega^n \Psi'(\omega)d\omega}{\Psi(\omega)-z},\qquad  1\leq \rho \leq h, 
\]
(One can take $\rho = 1$ and $\rho = h$ since the integrand extends continuously to the boundary by Caratheodory's theorem~\cite{Henrici}.) We note that the logarithmic derivative $\frac{d}{d\omega}\log(\Psi(\omega)-z)$ has a simple pole at $\omega_0$ with residue $-1$. If we set $G_z(\omega) = \omega \Psi'(\omega) / (\Psi(\omega) -z)$, then $G_z(\omega)$ can be written as the sum of a term of the form $(\omega-\omega_0)^{-1}$ and a doubly-infinite convergent Laurent series. By the residue theorem, we have
\[
\frac{1}{2\pi i}\int_{|\omega|=h}\frac{G_z(\omega)d\omega}{\omega^{k+1}} = a_k(z) - \omega_0^{-k},
\]
where
\[
G_z(\omega) = \frac{-\omega_0}{\omega - \omega_0} + \sum_{k=-\infty}^\infty a_k(z)\omega^k, \qquad a_k(z) =\begin{cases} \frac{1}{2\pi i}\int_{|\omega| = 1} \frac{G_z(\omega)d\omega}{\omega^{k+1}}, \\ \frac{1}{2\pi i}\int_{|\omega| = h} \frac{G_z(\omega)d\omega}{\omega^{k+1}} + \omega_0^{-k}. \end{cases}
\]
By doing the change-of-variables $\omega = e^{i\theta}$, we find that  
\[
\begin{aligned} 
a_{-n}(z) + \overline{a_{n}(z)} &=  \frac{1}{\pi}\!\int_{0}^{2\pi} \!\!{\rm Re}\!\left( G_z(e^{i\theta}) \right)\! e^{in\theta}d\theta, \\
a_{-n}(z)h^{-n} + \overline{a_{n}(z)}h^n &= \frac{1}{\pi}\!\int_{0}^{2\pi} \!\!{\rm Re}\!\left( G_z(he^{i\theta}) \right)\!e^{in\theta}d\theta + h^{-n}\omega_0^n + h^{n}\overline{\omega_0^{-n}}.
\end{aligned} 
\]
Since $R_n(z) = a_{-n}(z)$, we find that for $z\in E$ we have
\[
R_n(z) = \! \frac{1}{\pi(1-h^{-2n})} \!\!\int_0^{2\pi} \!\! e^{in\theta} \left( {\rm Re}(G_z(e^{i\theta})) - h^{-n}{\rm Re}(G_z(he^{i\theta}))\right)\!d\theta + \frac{h^{-2n}\omega_0^n + \overline{\omega_0}^{-n}}{1-h^{-2n}}.
\]
The geometric significance of this integrand is revealed by the identity:
\[
{\rm Re}\left(G_z(e^{i\theta})\right) = {\rm Im}\left(iG_z(e^{i\theta})\right) = \frac{d}{d\theta} {\rm Im} \!\left( \log(\Psi(e^{i\theta})-z) \right) = \frac{d}{d\theta} {\rm Arg}(\Psi(e^{i\theta}) - z).
\]
Therefore, we find that 
\begin{equation}\label{Total Arg Var}
    \Bigg|\int_0^{2\pi} e^{in\theta}{\rm Re}\left(G_z(e^{i\theta})\right)d\theta\Bigg| \leq \int_0^{2\pi} \!\left|\frac{d}{d\theta}{\rm Arg}\!\left(\Psi(e^{i\theta}) - z\right)\right| \!d\theta\leq 2\pi {\rm Rot}(E),
\end{equation}
where the last inequality follows since the total variation in argument around a closed curve as measured from a point not on the curve is bounded by total rotation~\cite[(6.14)]{Gaier_87_01}.  Similarly the same integral as~\cref{Total Arg Var} with integrand $e^{in\theta}{\rm Re}(G_z(he^{i\theta}))$ is bounded by $2\pi {\rm Rot}(F)$. The upper bound on $\sup_{z\in E} |R_n(z)|$ follows by noting that $1<|\omega_0|< h$. 
\end{proof} 

\Cref{lem:boundRn2} simplifies when $E$ and $F$ are, in addition, convex sets because we have ${\rm Rot}(E) = {\rm Rot}(F) = 1$. We obtain
\begin{equation} 
\sup_{z\in E} \left|R_n(z)\right| \leq \frac{3(1+h^{-n})}{1-h^{-2n}}, \quad n\geq 0. 
\label{eq:Convex2} 
\end{equation} 
Previously, it was shown by Ganelius that $\sup_{z\in E} \left|R_n(z)\right| \leq 4 e^2 n$~\cite{Ga}. For most practical $n$ and $h$, the bound in~\cref{lem:boundRn2} is sharper than the bound in~\cite{Ga}.  For example, for convex sets $E,F$,  the bound in~\cref{eq:Convex2} is an improvement over $4 e^2 n$ for all $n\geq 1$ if $h > 4e^2/(4e^2-3)\approx 1.113$. Similarly, when $h > 1.027$, the bound is an improvement for $n\geq 2$, and for any $n\geq 3$ when $h>1.012$ .  

There are opportunities to improve the bound in~\cref{lem:boundRn2} as~\cref{Total Arg Var} can be weak, especially when $h\approx 1$. The bound in~\cref{Total Arg Var} ignores potential cancellation in the integral $\int e^{in\theta} \frac{d}{d\theta} {\rm Arg}(\Psi - z) d\theta$. However, as the point $z$ approaches the boundary of $E$, the function $\frac{d}{d\theta} {\rm Arg}(\Psi - z)$ tends to a delta function centered at the value of $\theta$ corresponding to the limit on the boundary, which is $\pi$ if the boundary point is smooth. For this reason, we suspect that one can improve the bound in~\cref{lem:boundRn2} by a factor of about $2$.

By analytic continuation, the definition of $R_n$ can now be extended to $\Omega = \mathbb{C}\setminus (E\cup F)$. Fix $z \in \Omega$. First, we continuously deform the contour $\gamma$ to a contour $\gamma'$ that is contained in $\Omega$ and encircles $z$.  By continuously deforming the contour $\gamma'$ back to $\gamma$ plus a path traversed in both directions extending to an arbitrarily small circle around $z$, we find that 
\[
R_n(z) = \frac{1}{2\pi i} \int_{\gamma'} \frac{\Phi^n(\zeta)d\zeta}{\zeta - z} = \Phi^n(z) + \frac{1}{2\pi i} \int_\gamma \frac{\Phi^n(\zeta)d\zeta}{\zeta - z}, \qquad z \in \Omega. 
\]
Here, the term $\Phi^n(z)$ appears because it is the average value of the Cauchy integral over an arbitrarily small circle around $z$.  Since $|\Phi^n(z)|<h^n$ for $z\in\Omega$, we find that $R_n$ is a bounded function in $\Omega$. 

Since the Cauchy transform of a continuous function on a closed contour can be used to define two distinct holomorphic functions --- one in the interior of the region bounded by the contour and the other on the exterior --- we can write 
\[
\begin{aligned} 
\mathcal{C}_{\partial E}^+(\Phi^n)(z) &= \frac{1}{2\pi i} \!\! \int_\gamma \! \frac{\Phi^n(\zeta) d\zeta}{\zeta - z}, \qquad \text{$z$ inside of } \gamma,\\
\mathcal{C}_{\partial E}^-(\Phi^n)(z) &= \frac{1}{2\pi i} \!\! \int_\gamma \! \frac{\Phi^n(\zeta) d\zeta}{\zeta - z}, \qquad \text{$z$ outside of } \gamma,\\
\end{aligned} 
\]
where the subscript indicates that the integral is taken over the boundary of $E$. Therefore, the function $R_n(z)$ can be expressed as
\begin{equation}
    R_n(z) = \begin{cases} \mathcal{C}_{\partial E}^+(\Phi^n)(z),& z\in E,\\
                            \Phi^n(z) + \mathcal{C}_{\partial E}^-(\Phi^n)(z),& z \in \mathbb{C}\setminus( E\cup F). 
                            \end{cases}
                            \label{eq:RnDefinition}
\end{equation}

To further emphasize the interpretation that $R_n(z)$ is a filtered version of $\Phi^n(z)$, we show that $R_n(z)$ is relatively close to $\Phi^n(z)$ for $z\in\Omega$. 
\begin{lemma}\label{R - Phi bound} 
Let $E,F\in\mathbb{C}$ be sets satisfying the assumptions in~\cref{thm:MainTheorem}. Then, $R_n(z)$ in~\cref{eq:RnDefinition} satisfies
\[
\sup_{z\in \Omega} \left|R_n(z)-\Phi^n(z)\right| \leq 1 + \sup_{z\in E} \left|R_n(z)\right|.
\]
\label{lem:boundRn3}
\end{lemma} 
\begin{proof}
From the definition of $R_n(z)$ for $z\in\Omega$ (see~\cref{eq:RnDefinition}), we just need to bound $|\mathcal{C}_{\partial E}^-(\Phi^n)(z)|$.  Note that $\mathcal{C}^-_{\partial E}(\Phi^n)(z)$ is a bounded analytic function outside of $\gamma$ whose maximum modulus is attained on the curve $\gamma$. By the Sokhotski--Plemelj Theorem~\cite{Henrici} we find that 
 $\mathcal{C}^-_{\partial E}(\Phi^n)(z_0) = \mathcal{C}^+_{\partial E}(\Phi^n)(z_0) + 1$ for $z_0\in \partial E$. Therefore, $|\mathcal{C}^-_{\partial E}(\Phi^n)(z)| \leq \sup_{z\in E} \left|R_n(z)\right| + 1$.
\end{proof}

\Cref{lem:boundRn3} allows us to show that all the zeros of $R_n$ lie in $E$ or within a small neighborhood of $E$. Rouch\'{e}'s Theorem says that the winding numbers of $\Phi^n$ and $R_n$ around a closed curve $\Gamma$ will be equal provided that $|\Phi^n(z)-R_n(z)|<|\Phi^n(z)|$ for $z$ on $\Gamma$~\cite{Ahlfors}. By~\cref{R - Phi bound}, the theorem applies on any closed curve $\Gamma$ in $\Omega$ winding once around $E$ such that $1 + \sup_{z\in E} \left|R_n(z)\right|<|\Phi^n(z)|$ for $z$ on $\Gamma$. Such a curve $\Gamma$ can always be found when the bound $1 + \sup_{z\in E} \left|R_n(z)\right|<h^n$, say, by taking the image of $\Gamma$ to be an appropriate level set of $|\Phi^n|$. The map $\Phi^n$ has winding number precisely $n$ around $\Gamma$ by definition (though it is not defined in $E$) and hence so does $R_n$. Since $R_n$ is analytic inside $\Gamma$, it has $n$ zeros (counting multiplicities) inside $\Gamma$. Moreover, the same reasoning shows that $R_n$ has no additional zeros outside of $\Gamma$ in $\Omega$.

\begin{lemma} 
Let $E,F\in\mathbb{C}$ be sets satisfying the assumptions in~\cref{thm:MainTheorem}. Then, $R_n$ has precisely $n$ zeros (counting multiplicities) in the open set $\mathbb{C}\setminus F$. 
\end{lemma} 

We denote the distinct zeros of $R_n$ as $z_1,\ldots,z_K$ with corresponding multiplicities $m_1,\ldots,m_K$ such that $m_1+\cdots+m_K = n$.  In~\cite{Ga}, a more precise statement is proved about the location of the zeros of $R_n(z)$. For example, it is shown that for sufficiently large $n$, the zeros of $R_n(z)$ lie inside $E$ or in a neighborhood of $E$. We believe that if $E$ is a convex set, then the zeros of $R_n(z)$ lie in $E$ (as is known for Faber polynomials~\cite[Thm.~2]{KP}). 

\subsection{Step 2: Constructing a Faber rational function} \label{sec:general} 
\begin{figure} 
\centering 
\begin{minipage}{.44\textwidth}

\vspace{.1cm}

\begin{overpic}[width=\textwidth,trim= 2cm 1cm  1.55cm 1cm,clip=true]{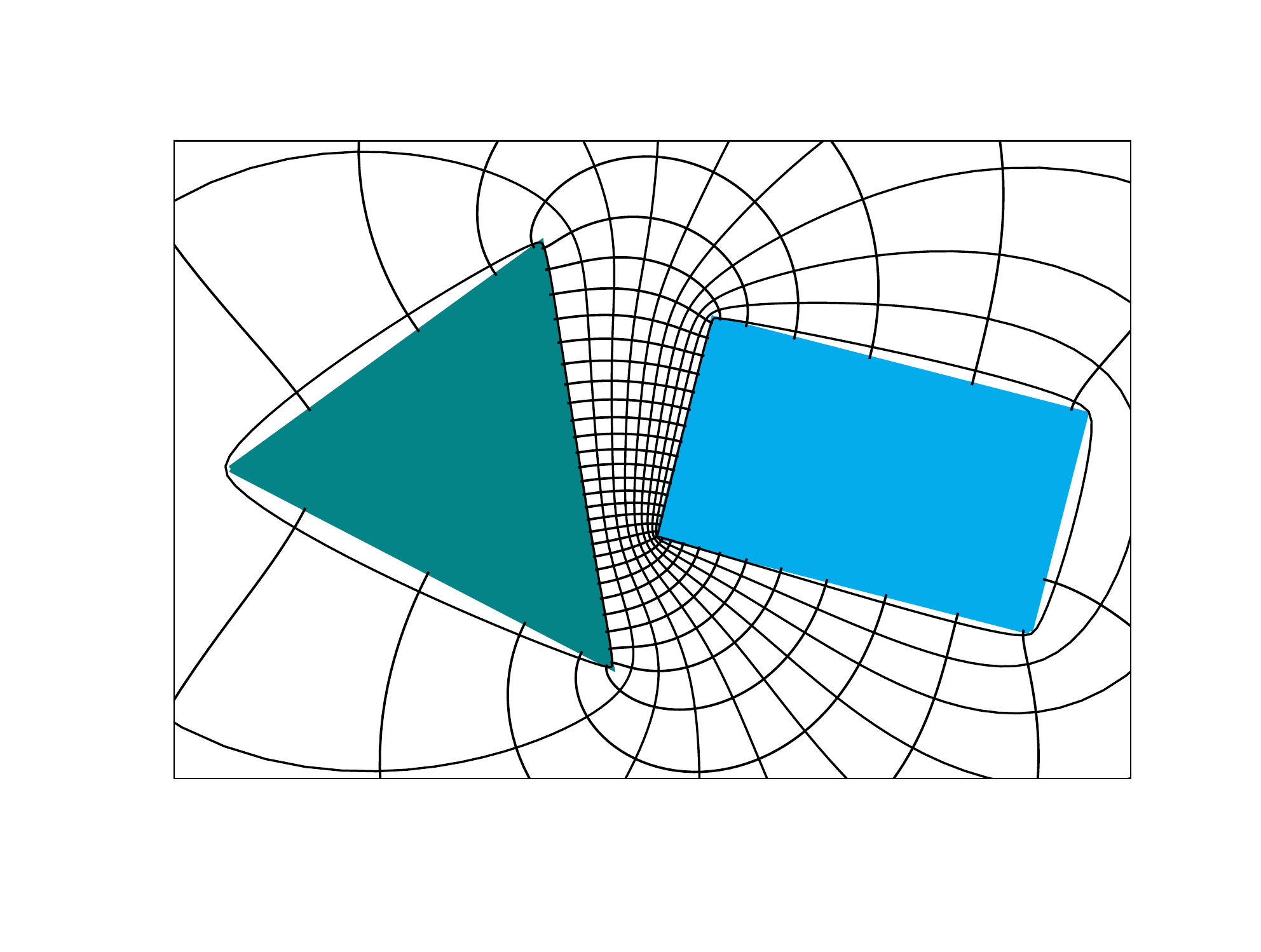}
\put(30, 39){E}
\put(68, 39){F}
\put(-2,70){\small{$i$}}
\put(-5,10){\small{$-i$}}
\put(-3,3){\small{$-1.5$}}
\put(92,3){\small{$1.5$}}
\put(45,3){\small{{\rm Re}(z)}}
\put(-3,35){\rotatebox{90}{\small{{\rm Im}(z)}}}
\end{overpic} 
\end{minipage} 
\begin{minipage}{.48\textwidth}
\begin{overpic}[width=\textwidth,trim= 0cm 0cm  0cm 0cm,clip=true]{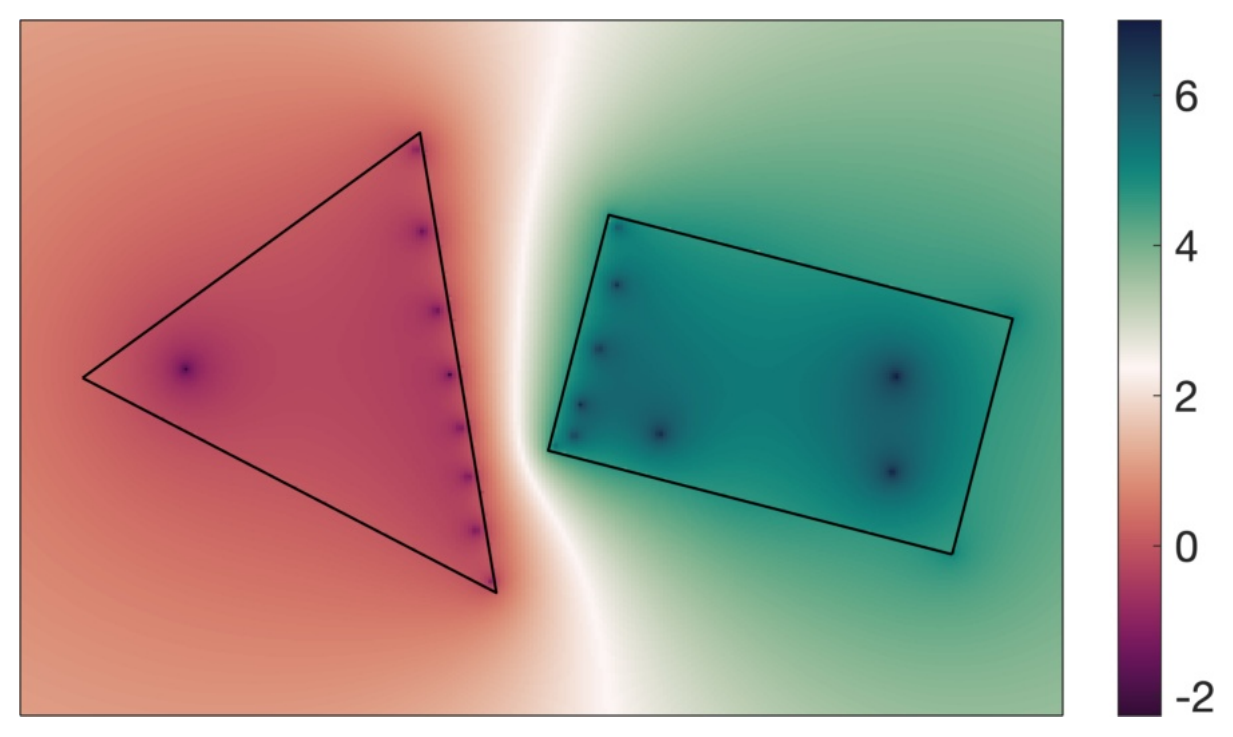}
\end{overpic} 
\end{minipage} 
\caption{Left: A plot of the conformal map $\Psi = \Phi^{-1} $, where $\Phi$ maps $\mathcal{C}\setminus E \cup F$ to the annulus $A = \{z\in\mathbb{C} : 1< |z|<h\}$.  Right: The magnitude of the type $(9,9)$ Faber rational function is plotted on a logarithmic scale. As $n$ increases, the Faber rational function grows increasingly larger on $F$ and smaller on $E$, making it useful for bounding the Zolotarev number $Z_n(E, F)$. }
\end{figure}
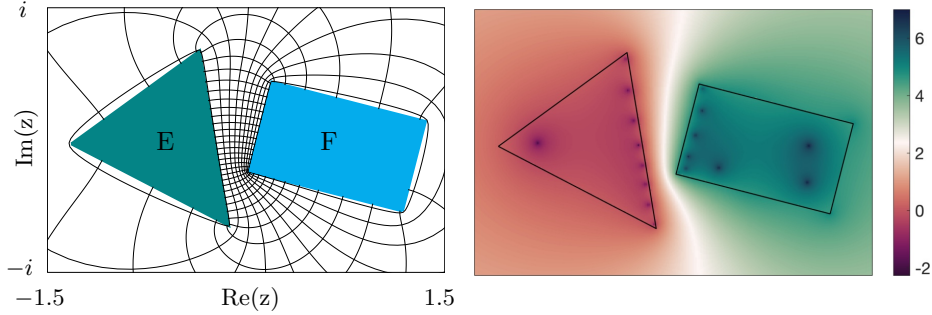 
While $R_n$ has precisely $n$ zeros, it is typically not a rational function. We must ``filter" $R_n$ again to obtain a rational function. 
 Let  $\eta:[0,1]\rightarrow \Omega$ be a curve that winds around $F$ once in the counterclockwise direction.
 We choose $\eta$ close to the boundary of $F$: Letting $0 < \delta < 1$,  $\eta$ satisfies $|\Phi(\eta(t))|\geq h-\delta$ for $t\in[0,1]$.\footnote{The requirement that $\delta>0$ is a technical necessity as $R_n(z)$ is defined for $z\in\mathbb{C}\setminus F$. Later, we take $\delta\rightarrow 0$ so conceptually one may prefer to think of $\eta$ as a parameterization of the boundary of $F$.}
By~\cref{R - Phi bound}, $R_n$ is close to $\Phi^n$ on $\eta$, and $|\Phi^n|$ is close to $h^n$ on $\eta$, so make sure that $\delta$ is sufficiently small to avoid encircling any zeros of $R_n$. Therefore, we can assume that $1/R_n$ is analytic on the curve $\eta$ (see~\cref{eq:RnDefinition}). We can construct analytic functions inside and outside of $\eta$ (the inside of $\eta$ contains $F$) as
\begin{equation} 
\begin{aligned} 
\mathcal{C}_{\eta}^+(1/R_n)(z) &= \frac{1}{2\pi i} \!\! \int_\eta \! \frac{d\zeta}{R_n(\zeta)(\zeta - z)}, \qquad \text{$z$ inside of } \eta,\\
\mathcal{C}_{\eta}^-(1/R_n)(z) &= \frac{1}{2\pi i} \!\! \int_\eta \! \frac{d\zeta}{R_n(\zeta)(\zeta - z)}, \qquad \text{$z$ outside of } \eta.\\
\end{aligned} 
\label{eq:CauchyIntegralsFaber} 
\end{equation} 
It is possible to give an exact expression for $\mathcal{C}_{\eta}^-(1/R_n)(z)$ in terms of $R_n(z)$ for $z$ outside of $\eta$. 
\begin{lemma} 
Let $E,F\in\mathbb{C}$ be sets satisfying the assumptions in~\cref{thm:MainTheorem} and $R_n(z)$ be defined as in~\cref{eq:RnDefinition}. If $z_1,\ldots,z_K$ are the distinct zeros of $R_n(z)$ with multiplicities $m_1+\cdots + m_K = n$, then for $z$ outside of $\eta$ we have
\begin{equation}\label{A2 C- expression}
    \mathcal{C}_{\eta}^-(1/R_n)(z) = -\frac{1}{R_n(z)} + \sum_{k=1}^K  \sum_{j=1}^{m_k}\frac{a^k_{-j}}{(z-z_k)^j} +  \frac{1}{R_n(\infty)}, 
\end{equation}
where $a^k_{-j}$ is the $z^{-j}$ coefficient of the principal part of the Laurent series for $R_n(z)$ about $z_k$. 
\label{lem:Cminus} 
\end{lemma} 
\begin{proof} 
We evaluate $\mathcal{C}_{\eta}^-(1/R_n)(z)$ on a large circle $\Gamma$ of radius $1/\Delta$ oriented clockwise enclosing $E$, $F$ and $z$, with detour paths in both directions leading to small counterclockwise circles around $z$, $\eta$, and each of the zeros of $R_n$, as well as the curve $\eta$ (see~\cref{fig:contour}). For an arbitrarily small $\epsilon >0$, we have 
\[
-\!\!\int_\Gamma  \!\frac{d\zeta}{R_n(\zeta)(\zeta - z)} \!= \!\!\int_{\eta}\!\frac{d\zeta}{R_n(\zeta)(\zeta - z)}  + \! \int_{|\zeta-z|=\epsilon}\!\!\frac{d\zeta}{R_n(\zeta)(\zeta - z)}  + \!\sum_{k=1}^K \!\int_{|\zeta-z_k|=\epsilon}\!\!\frac{d\zeta}{R_n(\zeta)(\zeta - z)}.
\]
If we perform the change-of-variables $\zeta = 1/t$ on the lefthand side, then we find that
 \[
 \begin{aligned}
    -\frac{1}{2\pi i}\!\int_\Gamma \! \frac{d\zeta}{R_n(\zeta)(\zeta - z)}\! =\! \frac{1}{2\pi i}\!\int_{|t|=\Delta}\!\frac{dt}{tR_n(\frac{1}{t})(1-zt)}\!
    = \!{\rm Res}_{t=0} \frac{1}{tR_n(\frac{1}{t})(1-zt)}
    = \!\frac{1}{R_n(\infty)}.
 \end{aligned}
 \]
 Since $|\Phi^n(\infty)|\leq h^n$, we note that $|R_n(\infty)|$ is bounded by~\cref{lem:boundRn3}.
For each circle of radius $\epsilon$ around $z_k$ for $1\leq k\leq K$, we find from the Residue Theorem that 
\[
   \lim_{\epsilon\rightarrow 0} \frac{1}{2\pi i} \int_{|\zeta-z_k|=\epsilon}\frac{d\zeta}{R_n(\zeta)(\zeta - z)} = -\sum_{j=1}^{m_i}\frac{a^i_{-j}}{(z-z_i)^j}. 
\]
These residues and the residue at the point $z$ are summed together to give~\cref{A2 C- expression}.
\end{proof} 
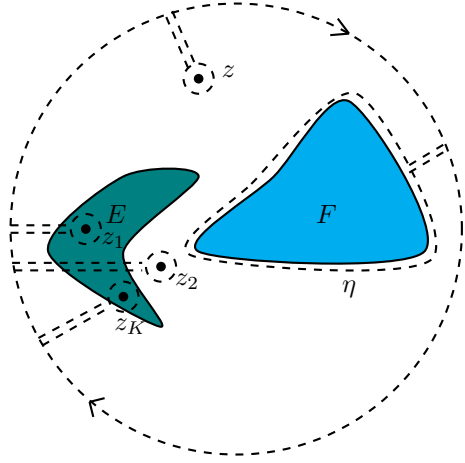
\begin{figure} 
\centering 
\begin{tikzpicture} 
\draw [thick, black, fill=teal] plot [smooth cycle] coordinates {(0,0) (1,1) (2,1) (1,0) (1.5,-1)};
\draw [thick, black, xshift=2cm,fill=cyan] plot [smooth cycle] coordinates {(0,0) (1,1) (2,2) (3,0)};
\draw [thick, dashed, black, xshift=2cm] plot [smooth cycle] coordinates {(-.1,-.05) (.7,.9) (2.1,2.1) (3.1,-.1)};
\draw [thick, dashed, black, xshift=6.5cm,yshift=.3cm] (-4,0) circle [radius=3];
\draw [thick, fill, black, xshift=6.5cm,yshift=.3cm] (-4.5,2) circle [radius=.05];
\draw [thick, fill, black, xshift=6.5cm,yshift=.3cm] (-5.5,-.9) circle [radius=.05];
\draw [thick, fill, black, xshift=6.5cm,yshift=.3cm] (-5,-.5) circle [radius=.05];
\draw [thick, fill, black, xshift=6.5cm,yshift=.3cm] (-6,0) circle [radius=.05];
\node at (.9,.5) {$E$};
\node at (3.7,.5) {$F$};
\draw [black,thick,dashed] (1.57,3.12) to (1.87,2.42);
\draw [black,thick,dashed] (1.65,3.2) to (1.95,2.5);
\draw [black,thick,dashed] (-.5,.25) to (.34,.25);
\draw [black,thick,dashed] (-.5,.35) to (.34,.35);
\draw [black,thick,dashed] (-.48,-.15) to (1.25,-.15);
\draw [black,thick,dashed] (-.45,-.25) to (1.25,-.25);
\draw [black,thick,dashed] (-.13,-1.15) to (.8,-.65);
\draw [black,thick,dashed] (-.08,-1.25) to (.8,-.75);
\draw [black,thick,dashed] (4.8,1.05) to (5.3,1.35);
\draw [black,thick,dashed] (4.8,1.15) to (5.3,1.45);
\draw [black,thick] (.55,-2) to (.6,-2.2);
\draw [black,thick] (.55,-2) to (.75,-1.95);
\draw [black,thick] (4,2.9) to (3.85,3.1);
\draw [black,thick] (4,2.9) to (3.75,2.88);
\draw [thick, dashed, black, xshift=6.5cm,yshift=.3cm] (-4.5,2) circle [radius=.2];
\draw [thick, dashed, black, xshift=6.5cm,yshift=.3cm] (-5.5,-.9) circle [radius=.2];
\draw [thick, dashed, black, xshift=6.5cm,yshift=.3cm] (-5,-.5) circle [radius=.2];
\draw [thick, dashed, black, xshift=6.5cm,yshift=.3cm] (-6,0) circle [radius=.2];
\node at (2.4,2.4) {$z$};
\node at (.87,.15) {$z_1$};
\node at (1.85,-.4) {$z_2$};
\node at (1.1,-1) {$z_K$};
\node at (4,-.5) {$\eta$};
\end{tikzpicture} 
\caption{The contour $\Gamma$ in the proof of~\cref{lem:Cminus}.}
\label{fig:contour}
\end{figure} 

\Cref{lem:Cminus} can be combined with the Sokhotski--Plemelj Theorem~\cite{Henrici} to find an expression for $\mathcal{C}_{\eta}^+(1/R_n)(z)$ in terms of $R_n(z)$. We have 
\[
\mathcal{C}_{\eta}^+(1/R_n)(z) - \mathcal{C}_{\eta}^-(1/R_n)(z) = \frac{1}{R_n(z)},\qquad \text{for $z$ on $\eta$}.
\]
and, by analytic continuation, we have 
\begin{equation}
\mathcal{C}_{\eta}^+(1/R_n)(z) = \sum_{k=1}^K \sum_{j=1}^{m_k}\frac{a^k_{-j}}{(z-z_k)^j} +  \frac{1}{R_n(\infty)}, \qquad \text{$z$ inside of } \eta.
\label{CplusDFA2}
\end{equation}
We conclude that $\mathcal{C}_{\eta}^+(1/R_n)$ is a rational function of type $(n,n)$. Finally, we define the Faber rational associated with $E$ and $F$ as
\begin{equation} 
\frac{1}{r_n(z)} = \sum_{k=1}^K \sum_{j=1}^{m_k}\frac{a^k_{-j}}{(z-z_k)^j} + \frac{1}{R_n(\infty)}.
\label{eq:FaberRationalFunction} 
\end{equation} 
The expression for $r_n(z)$ in~\cref{eq:FaberRationalFunction} is ideal for identifying $r_n(z)$ as a rational function of type $(n,n)$. The relationship $1/r_n(z) = \mathcal{C}_{\eta}^+(1/R_n)(z)$ in~\cref{eq:CauchyIntegralsFaber} is more convenient for practical computations as it does not involve computing $z_k$ for $1\leq k\leq K$ or the Laurent coefficients $\{a_k^{-j}\}$ (see~\cref{sec:Numeric}).

\section{Using the Faber rational to bound a Zolotarev number}\label{sec:TheBound} 
In this section we set out to find an upper bound on the Zolotarev number $Z_n(E,F)$ by bounding the Faber rational in~\cref{eq:FaberRationalFunction} associated with the sets $E$ and $F$. Since $s_n\in\mathcal{R}_{n,n}$ if and only if $1/s_n\in\mathcal{R}_{n,n}$, we note that $Z_n(E,F) = Z_n(F,E)$. Given our setup, we find it more convenient to derive a bound on $Z_n(E,F$) as follows:
\begin{equation} 
Z_n(E,F) = Z_n(F,E) = \inf_{s_n\in\mathcal{R}_{n,n}} \frac{\sup_{z\in F} |s_n(z)|}{\inf_{z\in E} |s_n(z)|} \leq \frac{\sup_{z\in F} |1/r_n(z)|}{\inf_{z\in E} |1/r_n(z)|},
\label{eq:ZolotarevBound} 
\end{equation} 
where $r_n$ is the Faber rational associated with $E$ and $F$. Therefore, we seek an upper bound on $\sup_{z\in F} |1/r_n(z)|$ in~\cref{sec:BoundingOnF} and a lower bound on $\inf_{z\in E} |1/r_n(z)|$ in~\cref{sec:BoundingOnE}.

\subsection{Bounding the Faber rational on $\mathbf{F}$}\label{sec:BoundingOnF}
From~\cref{CplusDFA2} and~\cref{eq:FaberRationalFunction}, we know that $1/r_n(z) = \mathcal{C}_{\eta}^+(1/R_n)(z)$ for $z\in F$.  Therefore, an upper bound on $\sup_{z\in F} |r_n(z)|$ follows from an upper bound on $\mathcal{C}_{\eta}^+(1/R_n)(z)$ for $z\in F$. From simple algebra, we have
\begin{equation}
\begin{aligned}
\mathcal{C}_{\eta}^+(1/R_n)(z) &= \frac{1}{2\pi i}\int_\eta \frac{d\zeta}{R_n(\zeta)(\zeta - z)}\\
&= \frac{1}{2\pi i } \int_\eta \frac{d\zeta}{\Phi^n(\zeta)(\zeta - z)} + \frac{1}{2\pi i} \int_\eta \frac{\Phi^n(\zeta) - R_n(\zeta)}{R_n(\zeta)\Phi^n(\zeta)}\frac{d\zeta}{\zeta-z}\\
&= \underbrace{\frac{-1}{2\pi i} \int_{|\omega|=h} \frac{1}{\omega^n}\frac{\Psi'(\omega)d\omega}{\Psi(\omega)-z}}_{=I(z)} + \underbrace{\frac{1}{2\pi i}\int_{|\omega|=h} \!\!\frac{\tilde{\varepsilon}(\omega)\Psi'(\omega)d\omega}{\Psi(\omega)-z}}_{=II(z)},
\end{aligned}
\label{two integrals1}
\end{equation} 
where $\tilde\varepsilon(\omega) = (\omega^n - R_n(\Psi(\omega)))/(R_n(\Psi(\omega))\omega^n)$. (One can take the contours over $|\omega| = h$ since the integrand extends continuously to the boundary by Caratheodory's theorem~\cite{Henrici}.)  Here, the minus sign in the definition of $I$ appears to respect the orientation of $\eta$ with respect to the interior of $\eta$. The integral $I$ may be bounded using the same argument as in the proof of~\cref{lem:boundRn2} to obtain 
\begin{equation} 
\sup_{z\in F} |I(z)| \leq \frac{M_n(F,E)}{1-h^{-2n}}h^{-n}, \quad M_n(F,E) = \left(2{\rm Rot}(F) + 2h^{-n}{\rm Rot}(E) + h^{-n}+1\right).
\label{eq:IntegralOne} 
\end{equation} 
To bound $|II(z)|$, we note that $\tilde{\varepsilon}(\omega)$ is holomorphic in the annulus $A$ with a pole at 0. So, for any $0<\alpha<1-1/h$, we have 
\[
|II(z)| \!=  \!\left|\frac{1}{2\pi i}\!\int_{|\omega|=(1-\alpha)h} \!\!\frac{\tilde{\varepsilon}(\omega)\Psi'(\omega)d\omega}{\Psi(\omega)-z}\right| \leq \sup_{|\omega| = (1-\alpha)h}\! |\tilde{\varepsilon}(\omega)| \frac{1}{2\pi}\!\int_{|\omega|=(1-\alpha)h} \bigg|\frac{\Psi'(\omega)d\omega}{\Psi(\omega)-z}\bigg|.
\]
By the same argument as Ganelius in~\cite[p.~411]{Ga} using the Koebe-$1/4$ Theorem, we find that 
\begin{equation}\label{G bound}
\frac{1}{2\pi}\int_{|\omega|=(1-\alpha)h} \bigg|\frac{\Psi'(\omega)d\omega}{\Psi(\omega)-z}\bigg| \leq \frac{4(1-\alpha)h}{d}, \qquad d = \min\{\alpha h, (1-\alpha)h-1\}.
\end{equation} 
The bound in~\cref{G bound} simplifies to $4(1-\alpha)/\alpha$ when $\alpha<(1-1/h)/2$. For the $\tilde{\varepsilon}$ term, we have 
\begin{equation}\label{ep bound}
\sup_{|\omega| = (1-\alpha)h} |\tilde{\varepsilon}(\omega)| \leq \frac{C_n}{(1-\alpha)^nh^n((1-\alpha)^nh^n-C_n)}, \qquad C_n = 1+\sup_{z\in E} |R_n(z)|
\end{equation}
as long as $((1-\alpha)^nh^n-C_n)>0$. We now want to find $0<\alpha<(1-1/h)/2$ with $((1-\alpha)^nh^n-C_n)>0$ to minimize the product of~\cref{G bound} and~\cref{ep bound} as that will derive a reasonable bound on $|\mathcal{C}_{\eta}^+(1/R_n)(z)|$ for $z\in F$. We have the following result:

\begin{lemma} 
For any $n > N_0$ with $N_0 = \max \{ 1 + 1/(h-1), \log(x_0)/\log(h)\}$, we have
\begin{equation}
    \min_{\substack{0<\alpha<(1-1/h)/2\\((1-\alpha)^nh^n-C_n)>0}} \frac{4(1-\alpha)C_n}{\alpha(1-\alpha)^nh^n((1-\alpha)^n h^n-C_n)} \leq \frac{32nh^n}{(h^n-C_n)^2(h^n+C_n)},
    \label{eq:minprob}
\end{equation}
where $x_0 = {\rm Rot}(E) + 1 + \sqrt{ ({\rm Rot}(E) + 1)^2 + 2{\rm Rot}(F)+1 }$.
\label{lem:MinBound}
\end{lemma} 
\begin{proof} 
Let $f(\alpha) = (4(1-\alpha)C_n)/(\alpha(1-\alpha)^nh^n((1-\alpha)^n h^n-C_n))$. Using calculus, we find that the minimum of~\cref{eq:minprob} is given by a unique value $0<\alpha_*<1/(2n)$ such that
\[
(1-\alpha_*)^nh^n = \frac{1-n\alpha_*}{1-2n\alpha_*} C_n \quad \Rightarrow \quad \alpha_* = \frac{h^n - \tfrac{1-n\alpha_*}{(1-\alpha_*)^n}C_n}{2nh^n}.
\]
Since $n>1+1/(h-1)$, we find that $\alpha_* <1/(2n)< (1-1/h)/2$.  Moreover, by using $(1-x)^n \geq 1- nx$ for $n\geq 1$ and $x\in\mathbb{R}$, we have $\alpha_* \geq \alpha_0$, where $\alpha_0 = (h^n-C_n)/(2nh^n)$. Note that we also have
\[
(1-\alpha_0)^nh^n-C_n \geq (1-n\alpha_0)h^n - C_n = \left(1-\frac{h^n-C_n}{2h^n}\right)h^n - C_n = \frac{h^n-C_n}{2}.
\]
Thus, by~\cref{lem:Bound1n}, we have that $(1-\alpha_0)^nh^n-C_n>0$ when $n > N_0$. Therefore, $\alpha_0$ satisfies the constraints in~\cref{eq:minprob} and we have the following upper bound on $f(\alpha_*)$: 
\begin{align*}
    f(\alpha_*) &\leq f\left(\frac{h^n-C_n}{2nh^n}\right) \leq \frac{4(1-\frac{h^n-C_n}{2nh^n})C_n}{\frac{h^n-C_n}{2nh^n}\frac{h^n+C_n}{2h^n}h^n(\frac{h^n+C_n}{2}-C_n)} = \frac{16((2n-1)h^n-C_n)C_n}{(h^n-C_n)^2(h^n+C_n)},
\end{align*}
where the second inequality follows from $(1-\alpha_0)^n \geq 1- n\alpha_0$. The result follows as $((2n-1)h^n-C_n)< 2nh^n$.
\end{proof} 
By combining~\cref{two integrals1},~\cref{eq:IntegralOne}, and~\cref{lem:MinBound}, we conclude that 
\begin{equation} 
\sup_{z\in F} \left|\frac{1}{r_n(z)}\right| \leq  \frac{M_n(F,E)}{1-h^{-2n}}h^{-n}+\frac{32nh^n}{(h^n-C_n)^2(h_n+C_n)},
\label{eq:BoundOnF}
\end{equation} 
where $M_n(F,E)$ is defined in~\cref{eq:IntegralOne}. The upper bound in~\cref{eq:BoundOnF} controls the numerator in~\cref{eq:ZolotarevBound}. 

\subsection{Bounding the Faber rational on $\mathbf{E}$}\label{sec:BoundingOnE}
From the triangle inequality, we have 
\[
\inf_{z\in E}\left|\frac{1}{r_n(z)}\right| \geq  \inf_{z\in E}\left|\frac{1}{R_n(z)}\right| - \sup_{z\in E}\left|\frac{1}{r_n(z)}-\frac{1}{R_n(z)}\right|,
\]
where $r_n$ is the Faber rational associated with the sets $E$ and $F$ and $R_n$ is defined in~\cref{eq:RnDefinition}. 
Since $R_n(z) = \mathcal{C}_{\partial E}^+(\Phi^n)(z)$ for $z\in E$, we have $1/r_n(z) = 1/R_n(z) +  \mathcal{C}^-_{\eta}(1/R_n)(z)$ for $z$ outside of $\eta$. 
Thus, we have 
\[
\inf_{z\in E}\left|\frac{1}{r_n(z)}\right|  \geq  \frac{1}{\sup_{z\in E}\left|\mathcal{C}_{\partial E}^+(\Phi^n)(z)\right|} - \sup_{z \in E}\left|\mathcal{C}^-_{\eta}(1/R_n)(z)\right|.
\]
\Cref{lem:boundRn2} provides an upper bound on the first term $\sup_{z\in E}\left|\mathcal{C}_{\partial E}^+(\Phi^n)(z)\right|$. For the second term, observe that $\mathcal{C}^-_{\eta}(1/R_n)$ is an analytic function outside of the contour $\eta$. Therefore, the maximum of $|\mathcal{C}^-_{\eta}(1/R_n)|$ is on the curve $\eta$, where we have the Sokhostski--Plemel Theorem. We find that for $z$ outside of the contour $\eta$,
\begin{align*}
|\mathcal{C}^-_{\eta}(1/R_n)(z)| &\leq \sup_{z\in \eta} |\mathcal{C}^+_{\eta}(1/R_n)(z)| + \sup_{z\in \eta}|1/R_n(z)|\\
&\leq \sup_{z\in \eta} |\mathcal{C}^+_{\eta}(1/R_n)(z)| + \frac{1}{(h-\delta)^n-C_n},\\
\end{align*}
where the second inequality follows from $|R_n(z)| \geq |\Phi^n(z)| - |R_n(z) - \Phi^n(z)|$, the fact that $|\Phi^n(z)| \geq (h-\delta)^n$ for $z$ on the curve $\eta$, and~\cref{lem:boundRn3}.  Because $0<\delta<1$ is arbitrarily small, we may take $|\mathcal{C}^-_{\eta}(1/R_n)(z)|\leq \sup_{z\in \eta} |\mathcal{C}^+_{\eta}(1/R_n)(z)| + 1/(h^n-C_n)$. Since $E$ is contained in the region outside of the contour $\eta$, we conclude that 
\begin{equation} 
\inf_{z\in E}\left|\frac{1}{r_n(z)}\right| \geq\frac{1-h^{-2n}}{M_n(E,F)} - \frac{M_n(F,E)}{1-h^{-2n}}h^{-n} - \frac{1}{h^n-C_n}.
\label{eq:BoundAboveE}
\end{equation} 
The lower bound in~\cref{eq:BoundAboveE} controls the denominator in~\cref{eq:ZolotarevBound}. As $n\rightarrow \infty$, this lower bound becomes $1/M_n(E,F)$. 

One has to be careful when using~\cref{eq:BoundAboveE} for small $n$, though, as the lower bound may happen to be negative, i.e., a trivial lower bound. To avoid this issue, we take 
\begin{equation} 
\inf_{z\in E}\left|\frac{1}{r_n(z)}\right| \geq \max\left\{0,\frac{1-h^{-2n}}{M_n(E,F)} - \frac{M_n(F,E)}{1-h^{-2n}}h^{-n} - \frac{1}{h^n-C_n} \right\}.
\label{eq:BoundAboveE2}
\end{equation} 
\subsection{Bounding the Zolotarev number}\label{sec:ZolotarevBound} 
Putting~\cref{eq:BoundOnF} and~\cref{eq:BoundAboveE2} together, completes the proof of~\cref{eq:FinalBound} and~\cref{thm:MainTheorem}.  That is, we have for $n>N_0$,
\[
Z_n(E,F) \leq \left( \frac{ \frac{M_n(E,F)M_n(F,E)}{1-h^{-2n}} + \frac{32n M_n(E,F)h^{-n}}{\left(1-C_nh^{-n}\right)^2(1+C_nh^{-n})}}{\max\left\{0,1 - \frac{M_n(E,F)M_n(F,E)}{1-h^{-2n}}h^{-n} - \frac{M_n(E,F)}{1-C_nh^{-n}}h^{-n} - h^{-2n}\right\}}\right)h^{-n}.
\]
To remove the dependency on $C_n$, note that $1+C_nh^{-n} \geq 1$ for all $n$. Further, by~\cref{lem:boundRn2} we have $1-C_nh^{-n} \geq 1 - (1+M_n(E,F))h^{-n}$, which is nonnegative for $n>N_0$ (see~\cref{lem:Bound1n}). This leads to our final explicit bound on $Z_n(E,F)$: 

\begin{theorem}[Main Theorem]
Let $E,F\subset \mathbb{C}$ be disjoint, simply-connected, compact sets with rectifiable Jordan boundaries. Then, for $h=\exp(1/{\rm cap}(E,F))$ and $n>N_0$, we have
\[
Z_n(E,F) \leq \!\!\left(\! \frac{ \frac{M_n(E,F)M_n(F,E)}{1-h^{-2n}} + \frac{32n M_n(E,F)h^{-n}}{(1 - (1+M_n(E,F))h^{-n})^2}}{\max\!\left\{0,1 - \frac{M_n(E,F)M_n(F,E)}{1-h^{-2n}}h^{-n} - \frac{M_n(E,F)}{1 - (1+M_n(E,F))h^{-n}}h^{-n} - h^{-2n}\right\}}\!\right)\!h^{-n}
\]
where $M_n(E,F) = 2{\rm Rot}(E) + 2h^{-n}{\rm Rot}(F) + h^{-n} +1$. Here, $N_0 = \max\{1+1/(h-1),\log(x_0)/\log(h)\}$, $x_0 = {\rm Rot}(E) + 1 + \sqrt{ ({\rm Rot}(E) + 1)^2 + 2{\rm Rot}(F)+1}$, and ${\rm Rot}(E)$ and ${\rm Rot}(F)$ are the total rotation of the boundaries of the domains $E$ and $F$, respectively.  
\label{thm:MainTheorem2} 
\end{theorem} 

The explicit bound in~\cref{thm:MainTheorem2} slightly simplifies when $E$ and $F$ are convex sets as ${\rm Rot}(E) = {\rm Rot}(F) = 1$. We find that $x_0 = 2+\sqrt{7}$ and $M_n(E,F) = M_n(F,E) = 3(1+h^{-n})$. We have
\[
Z_n(E,F) \leq \left( \frac{ \frac{9(1+h^{-n})^2}{1-h^{-2n}} + \frac{96n(1+h^{-n})h^{-n}}{(1 - 4h^{-n}-3h^{-2n})^2}}{\max\left\{0,1 - \frac{9(1+h^{-n})^2}{1-h^{-2n}}h^{-n} - \frac{3(1+h^{-n})}{1 - 4h^{-n}-3h^{-2n}}h^{-n} - h^{-2n}\right\}}\right)h^{-n},
\]
for any $n>\max\{1+1/(h-1),\log(2+\sqrt{7})/\log(h)\}$. 

\section{Other cases}\label{sec:Extensions} 
In addition to $E$ and $F$ both being compact sets, there are two other types of sets $E$ and $F$ that may be of interest to the reader. 

\subsection{$\mathbf{\mathbb{C}\setminus F}$ is a bounded domain containing $\mathbf{E}$}
\label{sec:A1case}
Let $E,F\subset \mathbb{C}$ be disjoint sets with rectifiable Jordan boundaries so that $\mathbb{C}\setminus F$ is an open and bounded domain containing a compact set $E$ (see~\cref{fig:SimpleConstruction1}). Small adjustments to our arguments in this paper are needed to bound $Z_n(E,F)$ here. In particular, due to the fact that $\Psi$ no longer has a pole in the annulus $A$,~\cref{lem:boundRn2} becomes 
\[
\sup_{z\in E} \left|R_n(z)\right| \leq  \frac{\tilde{M}_n(E,F)}{1-h^{-2n}}, \qquad \tilde{M}_n(E,F) = 2{\rm Rot}(E) + 2{\rm Rot}(F)h^{-n},
\]
which causes minor changes to the final bounds. We find that  
\begin{theorem} 
Let $E,F\subset \mathbb{C}$ be disjoint sets with rectifiable Jordan boundaries so that $\mathbb{C}\setminus F$ is open and $E$ is a compact subset of $\mathbb{C}\setminus F$. Then, for $h=\exp(1/{\rm cap}(E,F))$ and $n>\tilde{N}_0$, we have
\begin{equation} 
Z_n(E,F) \leq \!\!\left(\! \frac{ \frac{\tilde{M}_n(E,F)\tilde{M}_n(F,E)}{1-h^{-2n}} + \frac{32n \tilde{M}_n(E,F)h^{-n}}{(1 - (1+\tilde{M}_n(E,F))h^{-n})^2}}{\max\!\left\{0,1 - \frac{\tilde{M}_n(E,F)\tilde{M}_n(F,E)}{1-h^{-2n}}h^{-n} - \frac{\tilde{M}_n(E,F)}{1 - (1+\tilde{M}_n(E,F))h^{-n}}h^{-n} - h^{-2n}\right\}}\!\right)\!h^{-n}
\label{eq:bound3} 
\end{equation} 
where $\tilde{M}_n(E,F) = 2{\rm Rot}(E) + 2h^{-n}{\rm Rot}(F)$. Here, we also have $\tilde{N}_0 = \max\{1+1/(h-1),\log(\tilde{x}_0)/\log(h)\}$, $\tilde{x}_0 = {\rm Rot}(E) + 1/2 + \sqrt{ ({\rm Rot}(E) + 1/2)^2 + 2{\rm Rot}(F)}$, and ${\rm Rot}(E)$ and ${\rm Rot}(F)$ are the total rotation of the boundaries of the domains $E$ and $F$, respectively.  
\label{thm:MainTheorem3} 
\end{theorem} 

When the denominator in~\cref{eq:bound3} is zero, then one can take the trivial bound of $Z_n(E,F)\leq 1$ instead. For large $n$, we conclude that 
\[
1\leq \lim_{n\rightarrow \infty} \frac{Z_n(E,F)}{h^{-n}} \leq 4{\rm Rot}(E){\rm Rot}(F). 
\]
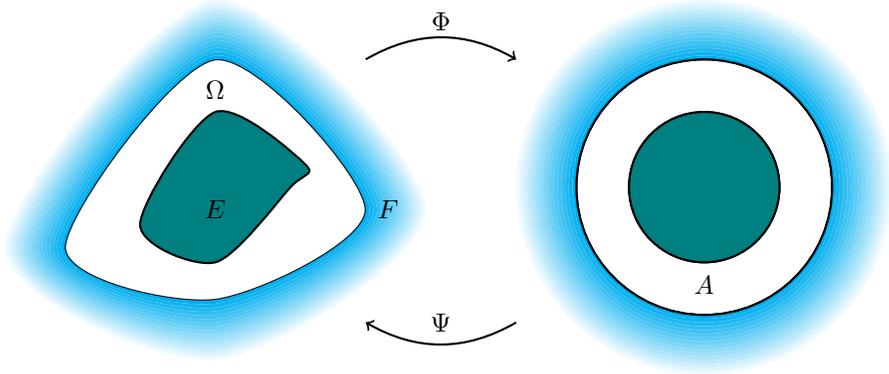
\begin{figure} 
\centering
\tikzfading[name=fade inside, inner color=transparent!0,outer color=transparent!50]
\begin{tikzpicture}
\draw [thick, fill=teal,yshift=.3cm] plot [smooth cycle] coordinates {(-1,-.5) (0,-1) (1,0) (1.2,.3) (0,1)};
\draw [thick, black, opacity=1, name path = one] plot [smooth cycle] coordinates {(-2,-.5) (0,-1.2) (2,0) (0,2)};
\draw [black, opacity=0, name path = two] plot [smooth cycle] coordinates {(-2.05,-.5) (0,-1.25) (2.05,0) (0,2.05)};
\draw [black, opacity=0, name path = c1] plot [smooth cycle] coordinates {(-2.1,-.5) (0,-1.3) (2.1,0) (0,2.1)};
\draw [black, opacity=0, name path = d1] plot [smooth cycle] coordinates {(-2.15,-.5) (0,-1.35) (2.15,0) (0,2.15)};
\draw [black, opacity=0, name path = e1] plot [smooth cycle] coordinates {(-2.2,-.5) (0,-1.4) (2.2,0) (0,2.2)};
\draw [black, opacity=0, name path = f1] plot [smooth cycle] coordinates {(-2.25,-.5) (0,-1.45) (2.25,0) (0,2.25)};
\draw [black, opacity=0, name path = g1] plot [smooth cycle] coordinates {(-2.3,-.5) (0,-1.5) (2.3,0) (0,2.3)};
\draw [black, opacity=0, name path = h1] plot [smooth cycle] coordinates {(-2.35,-.5) (0,-1.55) (2.35,0) (0,2.35)};
\draw [black, opacity=0, name path = a2] plot [smooth cycle] coordinates {(-2.4,-.5) (0,-1.6) (2.4,0) (0,2.4)};
\draw [black, opacity=0, name path = b2] plot [smooth cycle] coordinates {(-2.45,-.5) (0,-1.65) (2.45,0) (0,2.45)};
\draw [black, opacity=0, name path = c2] plot [smooth cycle] coordinates {(-2.5,-.5) (0,-1.7) (2.5,0) (0,2.5)};
\draw [black, opacity=0, name path = d2] plot [smooth cycle] coordinates {(-2.55,-.5) (0,-1.75) (2.55,0) (0,2.55)};
\draw [black, opacity=0, name path = e2] plot [smooth cycle] coordinates {(-2.6,-.5) (0,-1.8) (2.6,0) (0,2.6)};
\draw [black, opacity=0, name path = f2] plot [smooth cycle] coordinates {(-2.65,-.5) (0,-1.85) (2.65,0) (0,2.65)};
\draw [black, opacity=0, name path = g2] plot [smooth cycle] coordinates {(-2.7,-.5) (0,-1.9) (2.7,0) (0,2.7)};
\draw [black, opacity=0, name path = h2] plot [smooth cycle] coordinates {(-2.75,-.5) (0,-1.95) (2.75,0) (0,2.75)};
\draw [black, opacity=0, name path = a3] plot [smooth cycle] coordinates {(-2.8,-.5) (0,-2) (2.8,0) (0,2.8)};
\tikzfillbetween[of=one and two] {cyan, opacity=1};
\tikzfillbetween[of=two and c1] {cyan, opacity=.9373};
\tikzfillbetween[of=c1 and d1] {cyan, opacity=.875};
\tikzfillbetween[of=d1 and e1] {cyan, opacity=.8125};
\tikzfillbetween[of=e1 and f1] {cyan, opacity=.75};
\tikzfillbetween[of=f1 and g1] {cyan, opacity=.6875};
\tikzfillbetween[of=g1 and h1] {cyan, opacity=.625};
\tikzfillbetween[of=h1 and a2] {cyan, opacity=.5625};
\tikzfillbetween[of=a2 and b2] {cyan, opacity=.5};
\tikzfillbetween[of=b2 and c2] {cyan, opacity=.4375};
\tikzfillbetween[of=c2 and d2] {cyan, opacity=.375};
\tikzfillbetween[of=d2 and e2] {cyan, opacity=.3125};
\tikzfillbetween[of=e2 and f2] {cyan, opacity=.25};
\tikzfillbetween[of=f2 and g2] {cyan, opacity=.1875};
\tikzfillbetween[of=g2 and h2] {cyan, opacity=.125};
\tikzfillbetween[of=h2 and a3] {cyan, opacity=0.0625};
\draw [thick, black, xshift=6.5cm,yshift=.3cm] (0,0) circle [radius=1.7];
\draw [thick, black, xshift=6.5cm,fill=teal,yshift=.3cm] (0,0) circle [radius=1];
\node at (0,0) {$E$};
\node at (2.3,0) {$F$};
\node at (6.5,-1) {$A$};
\node at (0,1.6) {$\Omega$};
\draw [->,black,thick] (2,2) to [out=30,in=150] (4,2);
\draw [<-,black,thick] (2,-1.5) to [out=-30,in=-150] (4,-1.5);
\node at (3,-1.5) {$\Psi$};
\node at (3,2.5) {$\Phi$};
\fill [cyan,even odd rule,xshift=6.5cm,fill opacity=1,yshift=.3cm] (0,0) circle[radius=1.7cm] circle[radius=1.75 cm];
\fill [cyan,even odd rule,xshift=6.5cm,fill opacity=.9373,yshift=.3cm] (0,0) circle[radius=1.75cm] circle[radius=1.8 cm];
\fill [cyan,even odd rule,xshift=6.5cm,fill opacity=.875,yshift=.3cm] (0,0) circle[radius=1.8cm] circle[radius=1.85 cm];
\fill [cyan,even odd rule,xshift=6.5cm,fill opacity=.8125,yshift=.3cm] (0,0) circle[radius=1.85cm] circle[radius=1.9 cm];
\fill [cyan,even odd rule,xshift=6.5cm,fill opacity=.75,yshift=.3cm] (0,0) circle[radius=1.9cm] circle[radius=1.95 cm];
\fill [cyan,even odd rule,xshift=6.5cm,fill opacity=.6875,yshift=.3cm] (0,0) circle[radius=1.95cm] circle[radius=2 cm];
\fill [cyan,even odd rule,xshift=6.5cm,fill opacity=.625,yshift=.3cm] (0,0) circle[radius=2cm] circle[radius=2.05 cm];
\fill [cyan,even odd rule,xshift=6.5cm,fill opacity=.5625,yshift=.3cm] (0,0) circle[radius=2.05cm] circle[radius=2.1 cm];
\fill [cyan,even odd rule,xshift=6.5cm,fill opacity=.5,yshift=.3cm] (0,0) circle[radius=2.1cm] circle[radius=2.15 cm];
\fill [cyan,even odd rule,xshift=6.5cm,fill opacity=.4375,yshift=.3cm] (0,0) circle[radius=2.15cm] circle[radius=2.2 cm];
\fill [cyan,even odd rule,xshift=6.5cm,fill opacity=.375,yshift=.3cm] (0,0) circle[radius=2.2cm] circle[radius=2.25 cm];
\fill [cyan,even odd rule,xshift=6.5cm,fill opacity=.3125,yshift=.3cm] (0,0) circle[radius=2.25cm] circle[radius=2.3 cm];
\fill [cyan,even odd rule,xshift=6.5cm,fill opacity=.25,yshift=.3cm] (0,0) circle[radius=2.3cm] circle[radius=2.35 cm];
\fill [cyan,even odd rule,xshift=6.5cm,fill opacity=.1875,yshift=.3cm] (0,0) circle[radius=2.35cm] circle[radius=2.4 cm];
\fill [cyan,even odd rule,xshift=6.5cm,fill opacity=.125,yshift=.3cm] (0,0) circle[radius=2.4cm] circle[radius=2.45 cm];
\fill [cyan,even odd rule,xshift=6.5cm,fill opacity=0.0625,yshift=.3cm] (0,0) circle[radius=2.45cm] circle[radius=2.5 cm];
\draw [thick, black, xshift=6.5cm,yshift=.3cm] (0,0) circle [radius=1.7];
\draw [thick, black, xshift=6.5cm,fill=teal,yshift=.3cm] (0,0) circle [radius=1];
\end{tikzpicture}
\caption{Illustration of the typical setup when $\mathbb{C}\setminus F$ is a bounded domain containing a compact set $E$.}  
\label{fig:SimpleConstruction} 
\end{figure} 

\subsection{$\mathbf{F}$ is an unbounded domain and $\mathbf{E}$ is a compact domain contained in $\mathbf{\mathbb{C}\setminus F}$}
Let $E,F\subset \mathbb{C}$ be disjoint sets with rectifiable Jordan boundaries, where $F$ is an unbounded domain and $E$ is a compact domain contained in $\mathbb{C}\setminus F$ (see~\cref{fig:SimpleConstruction}). In this situation, our bound on $Z_n(E,F)$ is the same as that found in~\cref{thm:MainTheorem} and~\cref{eq:FinalBound}.  
\begin{figure} 
\centering
\tikzfading[name=fade inside, inner color=transparent!0,outer color=transparent!50]
\begin{tikzpicture}
\draw [thick, fill=teal,yshift=.3cm] plot [smooth cycle] coordinates {(-3,-.5) (-2,-1) (-1,0) (-1.1,.4) (-2,1)};
\draw [thick, black, opacity=1, name path = one] plot [smooth] coordinates {(1,-2.5) (0-.5,0) (1,2.5)};
\draw [black, opacity=0, name path = two] plot [smooth] coordinates {(1.05,-2.5) (0.05-.5,0) (1.05,2.5)};
\draw [black, opacity=0, name path = c1] plot [smooth] coordinates {(1.1,-2.5) (0.1-.5,0) (1.1,2.5)};
\draw [black, opacity=0, name path = d1] plot [smooth] coordinates {(1.15,-2.5) (0.15-.5,0) (1.15,2.5)};
\draw [black, opacity=0, name path = e1] plot [smooth] coordinates {(1.2,-2.5) (0.2-.5,0) (1.2,2.5)};
\draw [black, opacity=0, name path = f1] plot [smooth] coordinates {(1.25,-2.5) (0.25-.5,0) (1.25,2.5)};
\draw [black, opacity=0, name path = g1] plot [smooth] coordinates {(1.3,-2.5) (0.3-.5,0) (1.3,2.5)};
\draw [black, opacity=0, name path = h1] plot [smooth] coordinates {(1.35,-2.5) (0.35-.5,0) (1.35,2.5)};
\draw [black, opacity=0, name path = a2] plot [smooth] coordinates {(1.4,-2.5) (0.4-.5,0) (1.4,2.5)};
\draw [black, opacity=0, name path = b2] plot [smooth] coordinates {(1.45,-2.5) (0.45-.5,0) (1.45,2.5)};
\draw [black, opacity=0, name path = c2] plot [smooth] coordinates {(1.5,-2.5) (0.55-.5,0) (1.5,2.5)};
\draw [black, opacity=0, name path = d2] plot [smooth] coordinates {(1.55,-2.5) (0.6-.5,0) (1.55,2.5)};
\draw [black, opacity=0, name path = e2] plot [smooth] coordinates {(1.6,-2.5) (0.65-.5,0) (1.6,2.5)};
\draw [black, opacity=0, name path = f2] plot [smooth] coordinates {(1.65,-2.5) (0.7-.5,0) (1.65,2.5)};
\draw [black, opacity=0, name path = g2] plot [smooth] coordinates {(1.7,-2.5) (0.75-.5,0) (1.7,2.5)};
\draw [black, opacity=0, name path = h2] plot [smooth] coordinates {(1.75,-2.5) (0.8-.5,0) (1.75,2.5)};
\draw [black, opacity=0, name path = a3] plot [smooth] coordinates {(1.8,-2.5) (0.85-.5,0) (1.8,2.5)};
\tikzfillbetween[of=one and two] {cyan, opacity=1};
\tikzfillbetween[of=two and c1] {cyan, opacity=.9373};
\tikzfillbetween[of=c1 and d1] {cyan, opacity=.875};
\tikzfillbetween[of=d1 and e1] {cyan, opacity=.8125};
\tikzfillbetween[of=e1 and f1] {cyan, opacity=.75};
\tikzfillbetween[of=f1 and g1] {cyan, opacity=.6875};
\tikzfillbetween[of=g1 and h1] {cyan, opacity=.625};
\tikzfillbetween[of=h1 and a2] {cyan, opacity=.5625};
\tikzfillbetween[of=a2 and b2] {cyan, opacity=.5};
\tikzfillbetween[of=b2 and c2] {cyan, opacity=.4375};
\tikzfillbetween[of=c2 and d2] {cyan, opacity=.375};
\tikzfillbetween[of=d2 and e2] {cyan, opacity=.3125};
\tikzfillbetween[of=e2 and f2] {cyan, opacity=.25};
\tikzfillbetween[of=f2 and g2] {cyan, opacity=.1875};
\tikzfillbetween[of=g2 and h2] {cyan, opacity=.125};
\tikzfillbetween[of=h2 and a3] {cyan, opacity=0.0625};
\draw [thick, black, xshift=6.5cm,yshift=.3cm] (0,0) circle [radius=1.7];
\draw [thick, black, xshift=6.5cm,fill=teal,yshift=.3cm] (0,0) circle [radius=1];
\node at (-2,0) {$E$};
\node at (0,0) {$F$};
\node at (6.5,-1) {$A$};
\node at (-1,1.6) {$\Omega$};
\draw [->,black,thick] (2,2) to [out=30,in=150] (4,2);
\draw [<-,black,thick] (2,-1.5) to [out=-30,in=-150] (4,-1.5);
\node at (3,-1.5) {$\Psi$};
\node at (3,2.5) {$\Phi$};
\fill [cyan,even odd rule,xshift=6.5cm,fill opacity=1,yshift=.3cm] (0,0) circle[radius=1.7cm] circle[radius=1.75 cm];
\fill [cyan,even odd rule,xshift=6.5cm,fill opacity=.9373,yshift=.3cm] (0,0) circle[radius=1.75cm] circle[radius=1.8 cm];
\fill [cyan,even odd rule,xshift=6.5cm,fill opacity=.875,yshift=.3cm] (0,0) circle[radius=1.8cm] circle[radius=1.85 cm];
\fill [cyan,even odd rule,xshift=6.5cm,fill opacity=.8125,yshift=.3cm] (0,0) circle[radius=1.85cm] circle[radius=1.9 cm];
\fill [cyan,even odd rule,xshift=6.5cm,fill opacity=.75,yshift=.3cm] (0,0) circle[radius=1.9cm] circle[radius=1.95 cm];
\fill [cyan,even odd rule,xshift=6.5cm,fill opacity=.6875,yshift=.3cm] (0,0) circle[radius=1.95cm] circle[radius=2 cm];
\fill [cyan,even odd rule,xshift=6.5cm,fill opacity=.625,yshift=.3cm] (0,0) circle[radius=2cm] circle[radius=2.05 cm];
\fill [cyan,even odd rule,xshift=6.5cm,fill opacity=.5625,yshift=.3cm] (0,0) circle[radius=2.05cm] circle[radius=2.1 cm];
\fill [cyan,even odd rule,xshift=6.5cm,fill opacity=.5,yshift=.3cm] (0,0) circle[radius=2.1cm] circle[radius=2.15 cm];
\fill [cyan,even odd rule,xshift=6.5cm,fill opacity=.4375,yshift=.3cm] (0,0) circle[radius=2.15cm] circle[radius=2.2 cm];
\fill [cyan,even odd rule,xshift=6.5cm,fill opacity=.375,yshift=.3cm] (0,0) circle[radius=2.2cm] circle[radius=2.25 cm];
\fill [cyan,even odd rule,xshift=6.5cm,fill opacity=.3125,yshift=.3cm] (0,0) circle[radius=2.25cm] circle[radius=2.3 cm];
\fill [cyan,even odd rule,xshift=6.5cm,fill opacity=.25,yshift=.3cm] (0,0) circle[radius=2.3cm] circle[radius=2.35 cm];
\fill [cyan,even odd rule,xshift=6.5cm,fill opacity=.1875,yshift=.3cm] (0,0) circle[radius=2.35cm] circle[radius=2.4 cm];
\fill [cyan,even odd rule,xshift=6.5cm,fill opacity=.125,yshift=.3cm] (0,0) circle[radius=2.4cm] circle[radius=2.45 cm];
\fill [cyan,even odd rule,xshift=6.5cm,fill opacity=0.0625,yshift=.3cm] (0,0) circle[radius=2.45cm] circle[radius=2.5 cm];
\node at (7.3,1.6) {$\omega_0$};
\draw [thick, black,fill=black] (7,1.6) circle [radius=.05];
\draw [thick, black, xshift=6.5cm,yshift=.3cm] (0,0) circle [radius=1.7];
\draw [thick, black, xshift=6.5cm,fill=teal,yshift=.3cm] (0,0) circle [radius=1];
\end{tikzpicture}
\caption{Illustration of the typical setup when $F$ is an unbounded domain and $E$ is a compact domain contained in $\mathbb{C}\setminus F$. The location $\omega_0\in\mathbb{C}$ is the pole of the inverse map $\Psi = \Phi^{-1}$. }  
\label{fig:SimpleConstruction3} 
\end{figure}
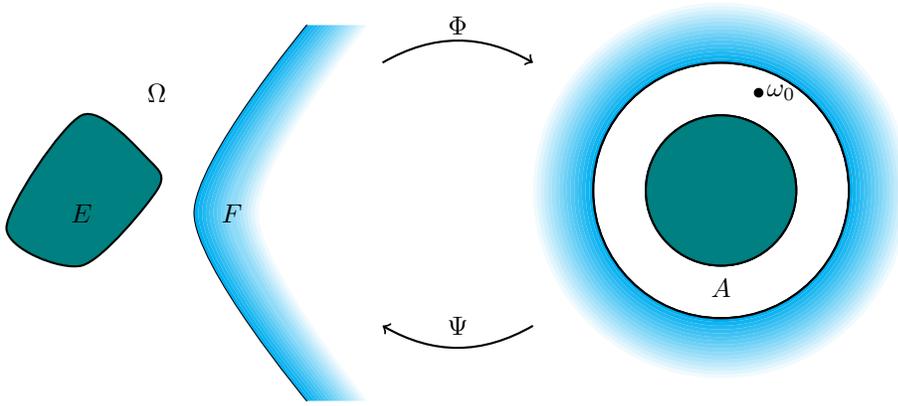 

\section{Numerical methods}
\label{sec:Numeric}
In this section, we briefly describe the algorithms we use to evaluate Faber rational functions, as well compute $h =  \exp(1/{\rm cap}(E,F))$. We also discuss a method for finding the poles and zeros of $r_n$.  

\subsection{Evaluating $\mathbf{r_n}$} To evaluate $r_n(z)$, we use the integral formulations for $1/r_n(z)$ developed in~\cref{sec:general}.  It is acceptable for numerical purposes to choose the contour $\eta$ in~\cref{lem:Cminus} as $\partial F$. Taking this liberty, we have from the lemma that
\begin{equation}
\label{eq:EvaluateFaber}
1/r_n(z)  =   \begin{cases}
\begin{aligned}
& \frac{1}{2\pi i} \! \int_{\partial F} \! \dfrac{d\zeta}{R_n(\zeta)(\zeta - z)}, \quad  &  z \in F, \phantom{\setminus F}  \\
  & -\frac{1}{R_n(z)}  + \dfrac{1}{2\pi i} \!\int_{\partial F}   \!  \dfrac{d\zeta}{R_n(\zeta)(\zeta - z)} , \quad   & z \in \mathbb{C} \setminus F,
  \end{aligned}
 \end{cases} 
 \end{equation}
where the first integral is understood in the principal value sense for  $z \in \partial F$,\footnote{We avoid sampling directly on $\partial F$ in our applications, and so omit discussion on the numerical computation of principle value integrals.} and $R_n$ is defined in~\cref{eq:RnDefinition}. 

The integrals in~\cref{eq:RnDefinition} and~\cref{eq:EvaluateFaber} can be computed using a quadrature rule, but these computations can become numerically unstable when $z$ is close to the contour of the integral being evaluated. To alleviate this issue, we apply a variant of the barycentric interpolation formula~\cite{berrut2004barycentric}. For $z \in F$, this takes the following form: 
\[ \frac{1}{r
_n(z)} = \dfrac{  {\displaystyle \int_{\partial F}} \! \dfrac{d\zeta}{R_n(\zeta)(\zeta - z)} }{ {\displaystyle \int_{\partial F}} \! \dfrac{d\zeta}{\zeta - z} } \approx \dfrac{ {\displaystyle \sum_{j = 1}^{N_Q}}  \dfrac{w_j}{R_n(x_j)(x_j - z)}}{ {\displaystyle \sum_{j = 1}^{N_Q}} \dfrac{w_j}{x_j - z}}, \]
where $\{(w_j, x_j)\}_{j = 1}^{N_Q}$ are an appropriate set of quadrature weights and nodes. A similar procedure is used when evaluating $R_n(z)$ for $z \in E$ near $\partial E$. Once $f_z = 1/r_n(z)$ is computed, we set $r_n(z) = 1/f_z$.  After one can evaluate $r_n$ on $E \cup F$, $r_n$ can be represented as a rational function via the AAA algorithm~\cite{nakatsukasa2018aaa}, which makes further evaluations more efficient. 

\subsection{Computing the conformal map}
\label{sec:ComputingCap}
Evaluating $R_n(z)$ requires the conformal map $\Phi: \Omega \to A$ (see~\cref{sec:ConformalMap}). We construct $\Phi$ using the method in~\cite{trefethen2020numerical}. In this approach, $\Phi$ is computed via the Green's function associated with the Laplacian operator on $\Omega$. The problem reduces to solving $\Delta u = 0$ with boundary conditions as in~\cite[p.~253]{Schiffer1950},~\cite[Sec.~4]{trefethen2020numerical}. To solve for $u$, boundary data is used to find the least squares fit to the coefficients of an approximate rational expansion of $u$. This is especially effective for resolving singularities in corners of the domain because the poles of the expansion are chosen to be exponentially clustered near the singular points~\cite{gopal2019}.  The modulus $h$ is treated as an additional unknown in the least squares system of equations, and it is recovered along with $u$. 

This method is versatile and can be used when $E$ and $F$ are polygons, as well as when their boundaries are either analytic curves or piecewise continuous analytic curves. It can be adapted for use in the case from~\cref{sec:A1case} where $F$ is unbounded.
\subsection{The poles and zeros of $\mathbf{r_n}$}
\label{sec:ComputePolesZeros}
To compute the poles and zeros of $r_n$, we first construct a representation of $r_n$ in barycentric form via the AAA algorithm~\cite{nakatsukasa2018aaa}.  This construction is computationally expensive because $r_n$ must be sufficiently sampled on the sets $E$ and $F$. The poles and zeros are then computed by solving an $(n+2) \times (n+2)$ generalized eigenvalue problem. To improve the accuracy of the computation, we apply AAA twice: first to $r_n$ on $E$ to compute the zeros, and then again to $r_n$ on $F$ to compute the poles. For an application involving poles and zeros, see~\cref{sec:ADI}. 

\section{Applications}
\label{sec:Applications}
We give two examples from numerical linear algebra where our results can be applied. In the first, we bound the singular values of Cauchy and Vandermonde matrices. In the second example, we treat $r_n$ as a proxy to the true infimal rational function that attains $Z_n(E, F)$. We show that the poles and zeros of $r_n$ are near-optimal parameters in the alternating direction implicit (ADI) method.  

\subsection{Bounding singular values of matrices with low displacement rank}

A matrix $X \in \mathbb{C}^{m \times p}$ is said to have  displacement rank $\nu$ if there are $A \in \mathbb{C}^{m \times m}$ and $B \in\mathbb{C}^{p \times p}$ so that ${\rm rank}(AX - XB) \leq \nu$.  When $A$ and $B$ are normal matrices with spectra $\lambda(A) \subset E$, $\lambda(B) \subset F$, the normalized singular values of $X$ are bounded  above in terms of Zolotarev numbers~\cite[Thm.~2.1]{Beckermann_19_01}. Specifically, 
\begin{equation}
\label{eq:GenSingularValueBound}
\sigma_{j \nu + 1}(X)\leq Z_j(E, F) \|X\|_2, \quad 0 \leq j\leq \lfloor (N-1)/\nu\rfloor, \quad N=\min(m,p), 
\end{equation}
where $\sigma_1(X) \geq \cdots \geq \sigma_N(X)$ are the nonzero singular values of $X$. 
Pairing this observation with~\cref{thm:MainTheorem,thm:MainTheorem2,thm:MainTheorem3} gives bounds on $\sigma_{j\nu + 1}(X)$  whenever $E$ and $F$ are as in the theorems. We illustrate the point with two examples.

\subsubsection{Complex-valued Cauchy matrices} Let $C$ be a Cauchy matrix in $\mathbb{C}^{m \times p}$,  with entries given by
$$C_{jk} = 1/(x_j - y_k), \quad  \underline{x} = \{x_j\}_{j =1}^{m} \subset E, \quad \underline{y} = \{y_k\}_{k =1}^{p} \subset F,$$
where $E$ and $F$ are as in \cref{thm:MainTheorem} and the sets $\underline{x}, \underline{y}$ are each collections of distinct points. Since ${\rm rank}(D_{\underline{x}}C-CD_{\underline{y}}) \leq 1$, where $D_{\underline{x}} = {\rm diag}(x_1, \ldots, x_m)$, it immediately follows from~\cref{eq:GenSingularValueBound} and~\cref{thm:MainTheorem} that for $0\leq j \leq N-1$, 
\[
\sigma_{j +1}(C) \leq K_{E, F} h^{-j} \|C\|_2, \qquad h = \exp(1/{\rm cap}(E,F)).
\]
Here, $K_{E, F}$ is given in~\cref{eq:FinalBound}. To implement the bound, we compute $h$ using the method in~\cref{sec:ComputingCap}. A comparison of the bounds to computed singular values is shown in~\cref{fig:SingularValueBounds}.

\begin{figure} 
\label{fig:SingularValueBounds}
\centering 
\begin{minipage}{.65\textwidth}
\begin{overpic}[width=\textwidth]{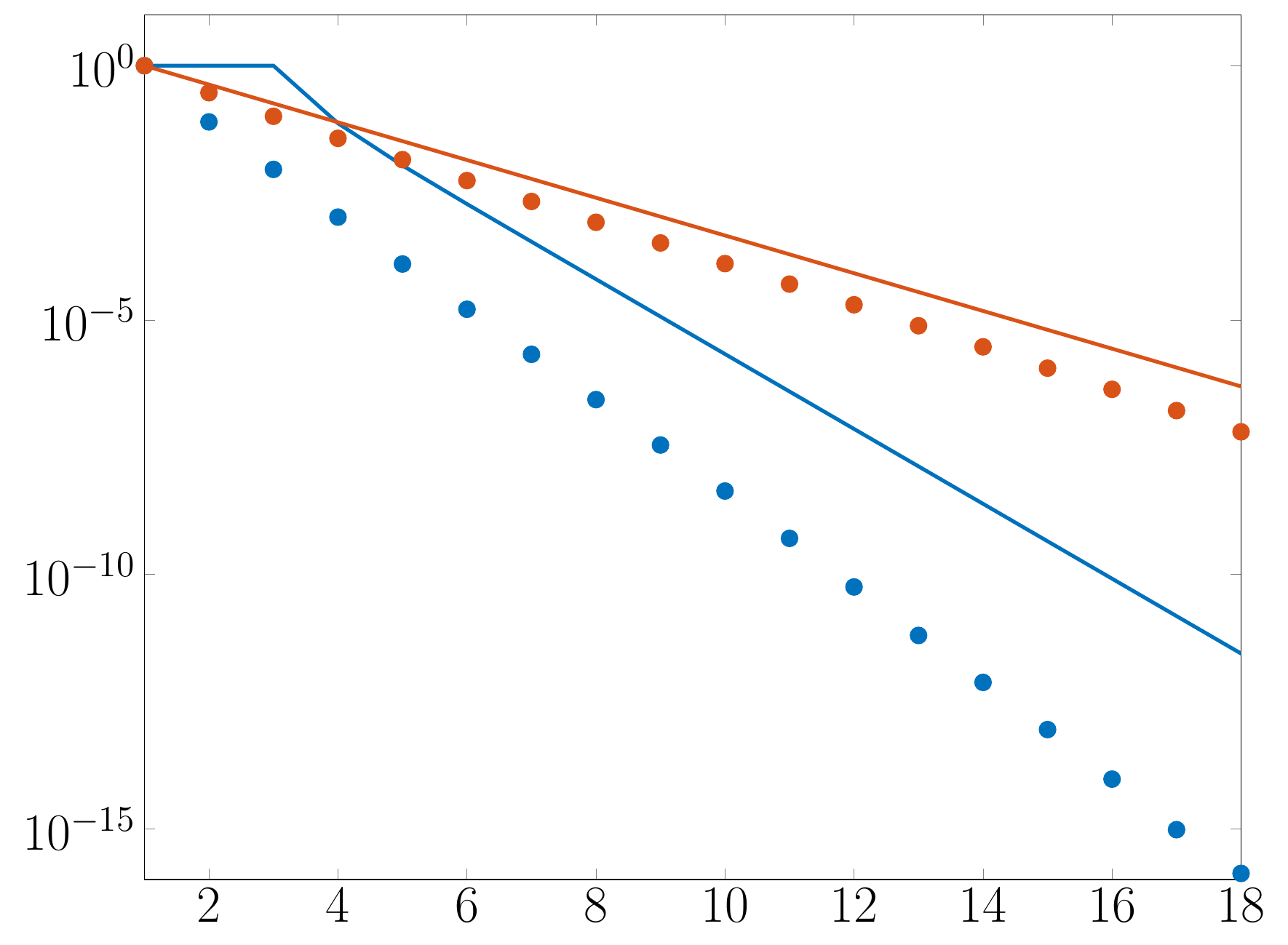}
\put(-3, 30){\rotatebox{90}{\small{magnitude}}}
\put(55, -3){\small{$j$}}
\put(64, 54){\rotatebox{-15}{\small{Vandermonde}}}
\put(77, 35){\rotatebox{-30}{\small{Cauchy}}}
\end{overpic} 
\end{minipage} 
\caption{ The first 18 normalized singular values of a Cauchy matrix (blue dots) and Vandermonde matrix (red dots) are plotted against the singular value index $j$ on a logarithmic scale. The Cauchy matrix is given by $C_{jk} = 1/(x_j -y_k)$, $1 \leq j,k \leq 100$, where for all $(j, k)$, $x_j \in E_C := \{ z\in \mathbb{C}: .3 \leq {\rm Re}(z) \leq 1.3, \, |{\rm Im}(z) |\leq .5\}$ and $y_k \in -E_C$. The nodes of the Vandermonde matrix $V \in \mathbb{C}^{100 \times 80}$ all lie in $E_V = \{ z\in \mathbb{C} : |z-(2+i)/10| < .4\}$. The solid lines show bounds on $\sigma_j(C)/\sigma_1(C)$ (blue) and $\sigma_j(V)/\sigma_1(V)$ (red)  obtained via~\cref{thm:MainTheorem} and~\cref{lemma:Vandermonde}, respectively. }
\end{figure}
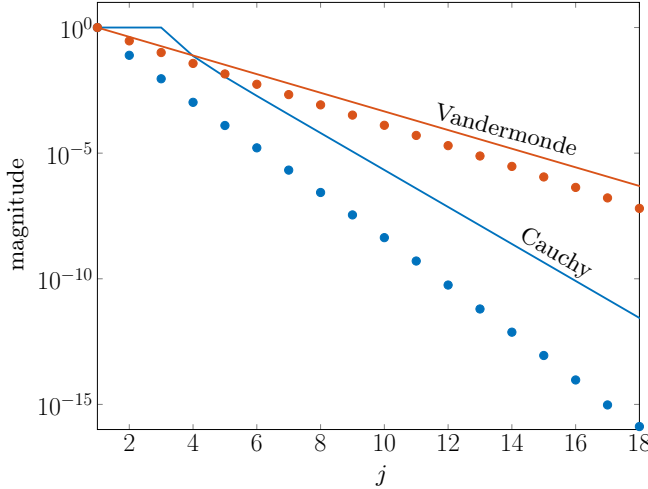

\subsubsection{Vandermonde matrices with nodes inside the unit circle} Let $V_\alpha$ be an $m \times p$ Vandermonde matrix with entries $(V_\alpha)_{jk} = \alpha_j^{(k-1)}$, where the nodes $\alpha =\{ \alpha_j\}_{j = 1}^m$ are distinct points in $\mathbb{C}$.  The singular values of $V_\alpha$ are known to decay rapidly when each $\alpha_j$ is real~\cite{Beckermann_19_01}, and there are multiple results on the (extremal) singular values of $V_\alpha$ when all  $|\alpha_j| = 1$~\cite{batenkov2019spectral,moitra2015super}. Less is known about singular value decay when $|\alpha_j| < 1$, despite the fact that this assumption is encountered in several applications~\cite{bazan2000conditioning, beylkin2005approximation,potts2010parameter}. 
We give the following lemma:   
\begin{lemma} 
\label{lemma:Vandermonde}
Let  $V_\alpha \in \mathbb{C}^{m \times p}$ have a set of distinct nodes  contained in the disk $ E:=\{ |z-z_0| < \eta_0\}$, $z_0 \neq 0$, where $E$ is in the open unit disk.  Then, the following bound holds for  $0 \leq j \leq N\!-\!1$, where $N = \min(m, p)$, :
\[
\sigma_{j+1}(V_\alpha) \leq h^{-j} \|V_\alpha\|_2, 
\]
where 
$$
h =\left|\frac{ z_0 - |{z}_0|\beta (z_0+\eta_0) }{ |{z}_0|(z_0+\eta_0)-\beta z_0} \right|, \quad \beta = \frac{1}{2|z_0|}\left(1 + c - \sqrt{(1+c)^2-4|z_0|^2}\right), \quad c = |z_0|^2 - \eta_0^2.
$$
\end{lemma}
\begin{proof}
We observe that ${\rm rank}( D_\alpha V-VQ) = 1$, where $Q = \left[ \begin{smallmatrix} 0 \phantom{-} & 1 \\ I_{n-1} & 0\end{smallmatrix} \right]$ is the circulant shift matrix. The eigenvalues $\lambda(Q)$ are the $p$th roots of unity. We choose $F$ as the set exterior to the open unit disk and note that $\lambda(Q) \subset F$. We map $\Omega:=C\setminus E\cup F$ to the annulus $A:= \{z\in \mathbb{C} : 1 < |z| < h\}$ with the following M\"obius transformation: 
$$ T(z):= \dfrac{h (|z_0|z-z_0 \beta )}{z_0 -|z_0| \beta z}, $$
where $h, \beta,$ and $c$ are as in the theorem. Since $T$ maps $\Omega \to A$ conformally and $T$ is rational,  $r_j = T^j$ is the rational function that attains $Z_j(E, F)$, and $Z_j(E,F)= h^{-j}$ (see~\cref{sec:Mobius}). Applying~\cref{eq:GenSingularValueBound} completes the proof. 
\end{proof}
The bounds from~\cref{lemma:Vandermonde} are shown in~\cref{fig:SingularValueBounds} along with computed singular values.  We remark that if $E$ is centered on the origin, then $h = 1/\eta_0$. For more general choices of $E$, an argument similar to the proof of~\cref{lemma:Vandermonde} can be applied using the bound on $Z_j(E, F)$ from~\cref{thm:MainTheorem2}.

\subsection{near-optimal ADI shift parameters}
\label{sec:ADI}

\begin{figure} 
\label{fig:ADI}
\centering 
\begin{minipage}{.49\textwidth}
\begin{overpic}[width=\textwidth,trim= 3cm 4cm 1cm 4cm]{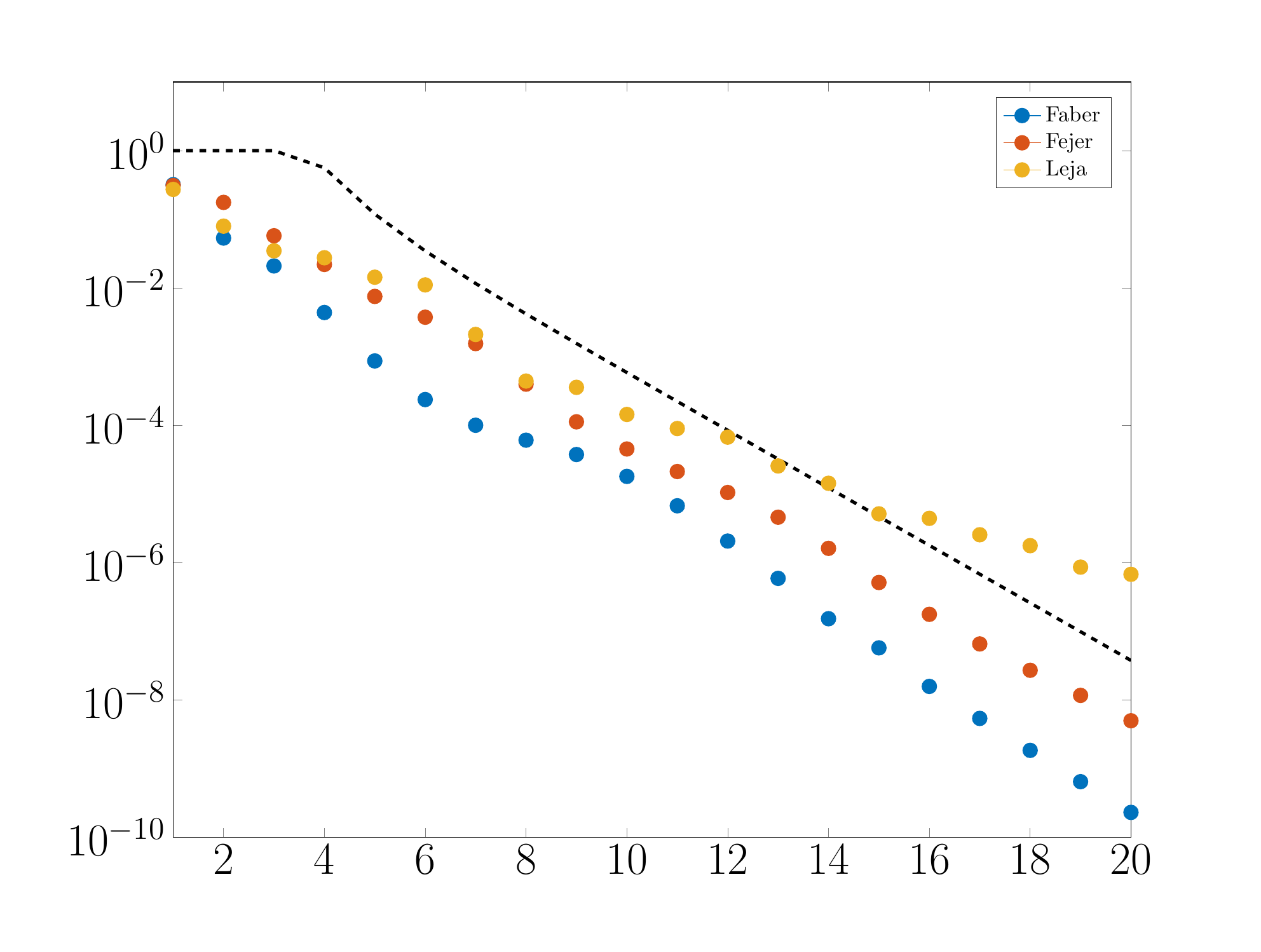}
\put(47, -22){\small{$k$}}
\put(-15, 17){\rotatebox{90}{\small{error}}}
\put(30,43){\rotatebox{-30}{\small{$K_{E,-E}h^{-k}$ }}}
\end{overpic} 
\end{minipage} 
\begin{minipage}{.34\textwidth}
\begin{overpic}[width=\textwidth]{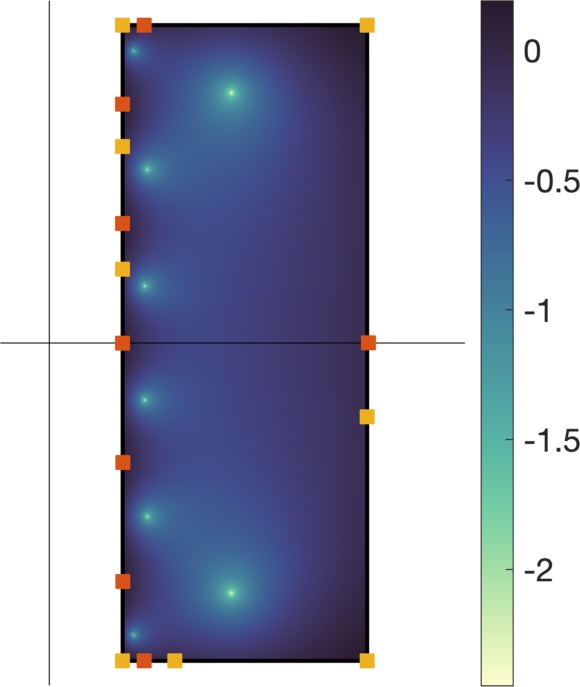}
\put(-8, 48){\small{${\rm Re}$}}
\put(3, 100){\small{${\rm Im}$}}
\end{overpic} 
\end{minipage} 
\caption{Left: The computed ADI error $\|X-X^{(k)}\|_2/ \|X\|_2$ is plotted against the indices $k$ on a logarithmic scale, where ADI is applied using Faber shifts (blue), generalized Fej\'{e}r points (red), and generalized Leja points (yellow). The bound on the error for ADI with Faber shifts is shown as a dotted line. Here, $X$ satisfies~\cref{eq:SylvesterADIeq}, with $m=p= 100$, $\lambda(A) \in E$, $\lambda(B) \in -E$, where $E = \{ z \in \mathbb{C}:  .3 \leq {\rm Re}(z)\leq 1.3,\, -1.3 \leq {\rm Im}(z) \leq 1.3 \}$. Right: The magnitude of the Faber rational $r_8$ is plotted on a logarithmic scale over $E$. Generalized Fej\'{e}r points (red squares) and generalized Leja points (yellow squares) associated with $(E, -E)$ are plotted. These are selected as the $\kappa_j$ parameters for ADI, and due to the symmetry of the domain, one sets  $\tau_j = -\overline{\kappa}_j$. The Faber shifts are formed by using the zeros of $r_8$ as $\kappa_j$ parameters, and the poles of $r_8$ (not depicted) as $\tau_j$ parameters. }
\end{figure}
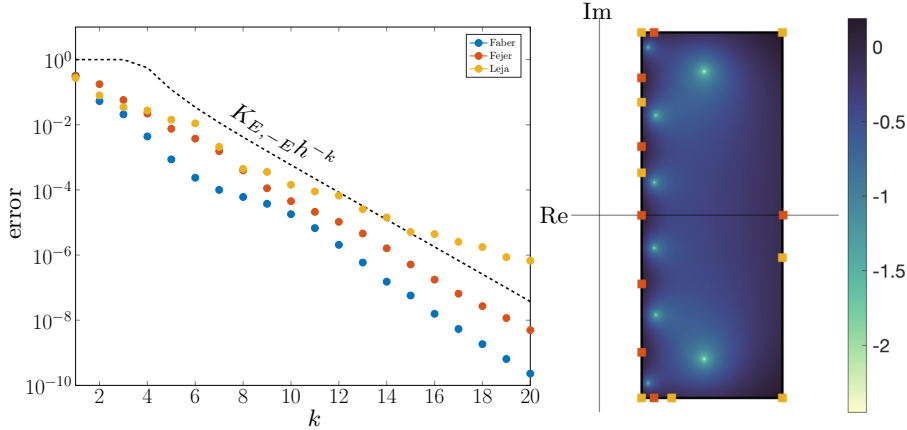

 The ADI method is an iterative method  used to solve the Sylvester matrix equation
\begin{equation}
\label{eq:SylvesterADIeq}
 AX - XB = M, \qquad X, M \in \mathbb{C}^{m \times p}.
 \end{equation}
For an overview with applications, see~\cite{simoncini2016computational}. One iteration of ADI consists of the following two steps: 
\begin{enumerate}
\item Solve for $X^{(j+1/2)}$, where 
\[
\left( A - \tau_{j+1} I  \right)X^{(j+1/2)} =   X^{(j)}\left( B-\tau_{j+1} I \right)  +   M. 
\]
\item Solve for $X^{(j+1)}$, where 
\[
X^{(j+1)}\left( B - \kappa_{j+1} I \right) = \left( A - \kappa_{j+1} I \right) X^{(j+1/2)}  - M.
\]
\end{enumerate}
The numbers $ (\kappa_j, \tau_j)$ are referred to as shift parameters, and an initial guess $X^{(0)} = 0$ is used to begin the iterations. After $k$ iterations, an approximate solution $X^{(k)}$ is constructed. Suppose that $A$ and $B$ are normal matrices with spectra $\lambda(A) \subset E$, $\lambda(B) \subset F$. Then, the ADI error is bounded by a rational function determined by the shift parameters~\cite{Lebedev_77_01}:
\begin{equation}
\label{eq:ADIerror}
\|X - X^{(k)}\|_2  \leq  \frac{\sup_{z \in E}|s_k(z)| }{\inf_{z \in F} |s_k(z)|} \|X\|_2 ,  \qquad  s_k(z) = \prod_{j = 1}^{k} 
\dfrac{ (z - \kappa_j)}{(z - \tau_j)}, \qquad k \geq 1.
\end{equation}
The bound in~\cref{eq:ADIerror} is minimized by selecting as shift parameters the poles and zeros of the rational function that attains $Z_k(E, F)$ in~\cref{eq:Zolotarev}. The solution to~\cref{eq:Zolotarev} is generally unknown, but when $E$ and $F$ are as in~\cref{thm:MainTheorem,thm:MainTheorem2,thm:MainTheorem3},  we choose $s_k$ as the Faber rational function $r_k$. We refer to the  poles and zeros of $r_k$ as  Faber shifts. The bounds on $Z_k(E, F)$ in~\cref{thm:MainTheorem,thm:MainTheorem2,thm:MainTheorem3} also bound the expression involving $s_k$ in~\cref{eq:ADIerror}. Since the bounds decay with $k$ at essentially the same rate as $Z_k(E,F)$, the Faber shifts are nearly optimal shift parameters.

We do not claim to have an efficient method for computing Faber shifts; the approach in~\cref{sec:ComputePolesZeros} is impractical for applications. For convex $E$, $F$, we observe that ADI with shifts derived from other so-called asymptotically optimal rational functions~\cite{Starke_92_01}, i.e., rationals $s_k$ with the property that 
\[
\lim_{k \to\infty} \left( \frac{\sup_{z\in E} |s_k(z)|}{\inf_{z\in F} |s_k(z)|}\right)^{1/k} =h^{-1},\qquad h=\exp\left(\frac{1}{{\rm cap}(E,F)}\right),
\]
often performs comparably to ADI with Faber shifts (see~\cref{fig:ADI}). This includes the generalized Fej\'{e}r points~\cite{walsh1965hyperbolic}, which can be computed with the inverse conformal map $\Psi$ from~\cref{sec:ConformalMap}, and the generalized Leja points, which are computed recursively by a greedy process~\cite{bagby1969interpolation, Starke_92_01}.

 \section*{Acknowledgements} 
 We thank Nick Trefethen for comments and suggestions on the manuscript. 

\appendix 
\section{Ensuring $\mathbf{n}$ is large enough for a valid bound}
In~\cref{lem:MinBound}, we need to show that for a sufficiently large $n$ we have $h^n>C_n$. 
\begin{lemma} 
For any integer $n>\log(x_0)/\log(h)$, we have $h^n>1 + M_n(E,F)$ (and hence, $h^n>C_n$), where 
\begin{equation}\label{eq:root}
    x_0 = {\rm Rot}(E) + 1 + \sqrt{ ({\rm Rot}(E) + 1)^2 + 2{\rm Rot}(F)+1 }. 
\end{equation}
\label{lem:Bound1n}
\end{lemma} 
\begin{proof} 
We know that $h^n>1+M_n(E,F)$ if 
\[
h^n > 2{\rm Rot}(E) + 2h^{-n}{\rm Rot}(F) + 2 + h^{-n}.
\]
This inequality is satisfied provided $h^n>x_0$, where $x_0$ is the largest root of the quadratic $x^2 - (2{\rm Rot}(E)+2)x - (2{\rm Rot}(F)+1)$. By the quadratic formula, $x_0$ is given in~\cref{eq:root}. Since $C_n = 1 + \sup_{z\in E} |R_n(z)|$, which is bounded above by $1+M_n(E,F)$ (see~\cref{lem:boundRn2}), we also have $h^n>C_n$ for $n>\log(x_0)/\log(h)$. 

\end{proof} 
This ensures that the optimization in the proof of~\cref{lem:MinBound} is valid. 


\end{document}